\documentclass[reqno]{amsart}
\usepackage{amsmath}
\usepackage{amsthm,amscd,amssymb,amsfonts, amsbsy}
\usepackage{latexsym}
\usepackage{mathrsfs}
\usepackage{pxfonts}
\usepackage{color}
\usepackage{enumerate}

\numberwithin{equation}{section}

\date{\today}

\theoremstyle{plain}
\newtheorem{theorem}{Theorem}[section]
\newtheorem{lemma}{Lemma}[section]
\newtheorem{corollary}{Corollary}[section]

\theoremstyle{definition}

\newtheorem*{acknowledgment}{Acknowledgment}

\theoremstyle{remark}
\newtheorem{remark}{Remark}[section]

\newcommand{\supp}{\operatorname{supp}}
\newcommand{\dist}{\operatorname{dist}}
\newcommand{\diam}{\operatorname{diam}}

\newcommand{\esssup}{\operatorname*{ess\,sup}}

\newcommand{\bR}{\mathbb R}

\newcommand{\bP}{\boldsymbol{P}}
\newcommand{\bQ}{\mathbf Q}

\newcommand\cB{\mathcal{B}}

\newcommand\cG{\mathcal{G}}

\newcommand\cX{\mathcal{X}}
\newcommand\cY{\mathcal{Y}}

\newcommand\sC{\mathscr{C}}

\newcommand\sE{\mathscr{E}}
\newcommand\sL{\mathscr{L}}

\newcommand\sP{\mathscr{P}}

\newcommand\sV{\mathscr{V}}

\newcommand\sW{\mathscr{W}}

\providecommand{\ip}[1]{\langle#1\rangle}
\providecommand{\Ip}[1]{\left\langle#1\right\rangle}
\providecommand{\set}[1]{\{#1\}}
\providecommand{\Set}[1]{\left\{#1\right\}}
\providecommand{\bigset}[1]{\bigl\{#1\bigr\}}

\providecommand{\abs}[1]{\lvert#1\rvert}
\providecommand{\Abs}[1]{\left\lvert#1\right\rvert}
\providecommand{\bigabs}[1]{\bigl\lvert#1\bigr\rvert}

\providecommand{\norm}[1]{\lVert#1\rVert}
\providecommand{\Norm}[1]{\left\lVert#1\right\rVert}

\providecommand{\tri}[1]{\lvert\!\lvert\!\lvert#1\rvert\!\rvert\!\rvert}

\renewcommand{\vec}[1]{\boldsymbol{#1}}
\renewcommand{\qedsymbol}{$\blacksquare$}

\begin{document}
\title[Green's function for Robin problem]{Green's functions for elliptic and parabolic systems with Robin-type boundary conditions}

\author[J. Choi]{Jongkeun Choi}
\address[J. Choi]{Department of Mathematics, Yonsei University, Seoul 120-749, Republic of Korea}
\email{cjg@yonsei.ac.kr}

\author[S. Kim]{Seick Kim}
\address[S. Kim]{Department of Mathematics, Yonsei University, Seoul 120-749, Republic of Korea}
\email{kimseick@yonsei.ac.kr}
\thanks{This work is partially supported by NRF Grant No. 2012R1A1A2040411.}

\begin{abstract}
The aim of this paper is to investigate Green's function for parabolic and elliptic systems satisfying a possibly nonlocal Robin-type boundary condition.
We construct Green's function for parabolic systems with time-dependent coefficients satisfying a possibly nonlocal Robin-type boundary condition assuming that weak solutions of the system are locally H\"{o}lder continuous in the interior of the domain, and as a corollary we construct Green's function for elliptic system with a Robin-type condition.
Also, we obtain Gaussian bound for Robin Green's function under an additional assumption that weak solutions of Robin problem are locally bounded up to the boundary.
We provide some examples satisfying such a local boundedness property, and thus have Gaussian bounds for their Green's functions.
\end{abstract}

\maketitle

\section{Introduction}
In this article, we are concerned with Green's functions for second-order elliptic and parabolic systems in divergence form subject to (possibly nonlocal) Robin-type boundary conditions.
Let $\Omega$ be a bounded Sobolev extension domain (e.g. a Lipschitz domain or a locally uniform domain) in $\bR^n$ ($n\ge 2$) and $Q=\Omega\times (a, b)$, where $-\infty \le a <b \le \infty$.
We consider the following parabolic operator
\[
\sL\vec u= \frac{\partial \vec u}{\partial t}- \sum_{\alpha, \beta =1}^n  \frac{\partial}{\partial x_\alpha} \left( A^{\alpha\beta}\frac{\partial \vec u}{\partial x_\beta}\right)
\]
acting on a vector valued function $\vec u=(u^1,\ldots, u^m)^\top$ defined on $Q$, where
the $A^{\alpha\beta}=A^{\alpha\beta}(x,t)$ are $m\times m$ matrix valued functions on $\bR^{n+1}$ with entries $a^{\alpha\beta}_{ij}$ satisfying the  uniform strong ellipticity condition; i.e., there is a constant $\lambda \in (0,1]$ such that for all $(x,t)\in \bR^{n+1}$ and any vectors $\vec \xi=(\xi^i_\alpha)$ and $\vec \eta=(\eta^i_\alpha)$ in $\bR^{mn}$, we have
\begin{equation}
\label{B2.eq1}
\sum_{\alpha,\beta=1}^n \sum_{i,j=1}^m a^{\alpha\beta}_{ij} \xi^j_\beta \xi^i_\alpha \ge \lambda \abs{\vec \xi}^2, \quad \Abs{\sum_{\alpha,\beta=1}^n \sum_{i,j=1}^m a^{\alpha\beta}_{ij}\xi^j_\beta \eta^i_\alpha}\le \lambda^{-1}\abs{\vec \xi}\abs{\vec \eta}.
\end{equation}
The adjoint operator $\sL^{*}$ of  $\sL$ is given by
\[
\sL^{*} \vec u=  -\frac{\partial \vec u}{\partial t} -\frac{\partial}{\partial x_\alpha} \left(A^{\beta\alpha}(x,t)^\top \frac{\partial \vec u}{\partial x_\beta}\right).
\]
We shall assume that the operators $\sL$ and $\sL^{*}$ has a property such that weak solutions of $\sL\vec u =0$ or $\sL^{*} \vec u=0$ are locally H\"older continuous in the interior of the domain.
In fact, this property is satisfied by a large class of operators.
For example, a celebrated theorem by J. Nash shows that this property always holds when $m=1$.
Other examples include the case when the coefficients of $\sL$ are uniformly continuous or, more generally, belong to VMO in $x$-variable.
The (possibly nonlocal) Robin-type boundary conditions are formally of the type
\begin{equation}			\label{eq0.02i}
\left(\frac{\partial \vec u}{\partial \nu} +\Theta(t) \vec u \right) \Bigg|_{\partial\Omega}= 0\quad \text{for }\; t \in  (a, b),
\end{equation}
where $\partial /\partial \nu$ is the usual (outward) co-normal derivative
\[
\frac{\partial}{\partial \nu}=\sum_{\alpha,\beta =1}^ n n_\alpha  A^{\alpha\beta} \frac{\partial}{\partial x_\beta}
\]
and $\Theta=\Theta(t)$ acts in appropriate Sobolev spaces on the boundary $\partial\Omega$.
We shall also require that $\Theta$ is non-degenerate in certain sense (see Section~\ref{B} below).

Our investigation is largely motivated by a very recent interesting article by F. Gesztesy, M. Mitrea, and R. Nichols \cite{GMN}, where the authors showed, among other things, Gaussian heat kernel bounds assuming $\Theta \ge 0$ in the scalar case (i.e. $m=1$).
Their argument is based on a careful analysis on the resolvent and semigroup of a self-adjoint realization of the corresponding elliptic operator in $L^2(\Omega)$.
In this article, we follow an approach that is different from theirs and based on techniques developed in recent papers \cite{CDK, CDK2, CK2}.
We construct Green's function for $\sL$ in $\Omega\times(-\infty,\infty)$ satisfying the Robin boundary condition \eqref{eq0.02i}; i.e. an $m\times m$ matrix valued function $\vec  \cG(x,t,y,s)$ that is, as a function of $(x,t)$ with $(y,s)$ fixed, a generalized solution of the problem
\[
\left\{
\begin{array}{cl}
\sL \vec  \cG(x,t,y,s) = 0 \qquad & \text{in }\; \Omega \times (s,\infty),\\
\frac{\partial}{\partial \nu} \vec \cG(x,t,y,s) +\Theta(t)\vec  \cG(x,t,y,s) = 0
 &  \text{on }\; \partial\Omega\times (s,\infty),\\
\vec  \cG(x,t,y,s)\equiv \delta_y(x) \vec I \quad& \text{on }\; \Omega\times \set{t=s},
\end{array}
\right. 
\]
where $\delta_{y}(\cdot)$ is Dirac delta function concentrated at $y$ and $\vec I$ is the $m\times m$ identity matrix.
More precise definition of Robin Green's function is given in Section~\ref{2.rp}.
In the case when the coefficients are time-independent,  $\vec K(x,y,t):= \vec  \cG(x,t,y,0)$ is called a Robin heat kernel.
By using a Robin heat kernel, we construct Green's function $\vec G(x,y)$ for elliptic systems satisfying Robin boundary conditions.
Also, we are interested in  the following global Gaussian estimate for the Robin Green's function:
There exist positive constants $C$ and $\kappa$ such that for all $t>s$ and $x,y \in \Omega$, we have
\begin{equation}        \label{eq1.05a}
\abs{\vec \cG (x,t,y,s)} \le \frac{C}{\min\left\{\sqrt{t-s}, \diam \Omega\right\}^{n}}\exp\Set{-\frac{\kappa\abs{x-y}^2}{t-s}}.
\end{equation}
If we assume further that the operator $\sL$ has the property that weak solutions of $\sL \vec u=0$ in $Q$ with zero Robin data on the lateral boundary are locally bounded, then we show that the Robin Green's function has the Gaussian upper bound \eqref{eq1.05a}.
We show that this local boundedness property is, for example, satisfied when
\begin{enumerate}[i)]
\item
$m=1$ (the scalar case) and $\Theta=M_\theta$, the operator of multiplication with a nonnegative measurable function $\theta$ on $\partial\Omega\times (-\infty,\infty)$ that belongs to a suitable Lebesgue class.

\item
$\Omega$ is a Lipschitz domain in $\bR^2$, $\Theta= M_{\vec \theta}$, where $\theta$ is an $m\times m$ matrix-valued $L^\infty$-function, and the coefficients of $\sL$ and $\theta$ are $t$-independent.

\item
$\Omega$ is a $C^1$ domain in $\bR^n$ ($n\ge 3$), the coefficients of $\sL$ are uniformly VMO in $x$, and $\Theta= M_{\vec \theta}$, where $\theta$ is an $m\times m$ matrix-valued $L^\infty$-function.
\end{enumerate}
By using the Gaussian estimate \eqref{eq1.05a}, we prove that elliptic Robin Green's function has the global bound
\begin{enumerate}[i)]
\item
$n=2$
\[
\abs{\vec G(x,y)} \le C\left\{1+\ln \left(\frac{\diam \Omega}{\abs{x-y}}\right)\right\},
\]

\item
$n\ge 3$
\[
\abs{\vec G(x,y)} \le C\abs{x-y}^{2-n}.
\]
\end{enumerate}

The Green's function for a scalar parabolic equation with real measurable coefficients in the free space was first studied by Nash \cite{Nash} and its two-sided Gaussian bounds were obtained by Aronson \cite{Aronson}.
There are vast literature regarding heat kernels for second order elliptic operators satisfying Dirichlet or Neumann boundary conditions.
We find it very hard to list them all here and just refer to monographs by Davies \cite{Davies89}, Ouhabaz \cite{Ouhabaz05}, and references therein.
We also mention related monographs by Robinson \cite{Robinson}, Grigor'yan \cite{Gri}, and Gyrya and Saloff-Coste \cite{GSC11}.
It is well known that Aronson's bounds are no longer available for a parabolic equation with complex valued coefficient when $n\ge 3$.
Auscher \cite{Auscher} obtained an upper Gaussian bound of the heat kernel for an elliptic operator whose coefficients are complex $L^\infty$-perturbation of real coefficients; see \cite{AMcT, AT, AT2}.
Hofmann and Kim \cite{HK04} extended Auscher's result to parabolic systems with time-dependent coefficients.
For heat kernels satisfying Robin conditions, we already mentioned the paper by Gesztesy, Mitrea, and Nichols \cite{GMN}, where one can find a survey of literature devoted to Robin boundary conditions, among which we particularly mention papers by Arendt and ter Elst \cite{AtE} and by Daners \cite{Daners}.

The novelty of this paper in constructing Green's function lies in that we allow the operators to have time-dependent coefficients and that our domains are more general than Lipschitz domains.
Also, we obtain Gaussian upper bounds of Robin Green's functions for parabolic systems as well as for scalar equation.
In Gesztesy et al. \cite{GMN}, the authors also considered parabolic systems, but Gaussian bounds were established only for the scalar case.
Moreover, as an important application of our result on the Gaussian bounds, we obtain the usual bounds for elliptic Green's function.
Especially, in the two dimensional case, we get a logarithmic bound for Green's function of elliptic systems satisfying a pointwise Robin condition without assuming any regularity on the coefficients, and we believe this is new.
In our paper, the key for obtaining Gaussian bounds lies in establishing local boundedness property for weak solutions and we allocate a large portion of the paper to prove local boundedness properties for the above mentioned three special but important cases.
We mention that the approach adopted in this paper is similar to that developed in recent papers \cite{CDK, CDK2, CK2}, where Green's functions for time-dependent parabolic systems subject to Dirichlet or Neumann condition were investigated with almost minimal assumptions on the coefficients and domains.
It seems to us that there is no literature dealing with Green's function satisfying Robin boundary condition for time-dependent parabolic systems in such generality, and we hope that this paper contributes  towards filling the gap and serves as a reference.

The organization of paper is as follows.
In Section~\ref{B}, we introduce notation and definitions used in this paper.
We mostly use the same the notation as in \cite{CK2} to help readers because we frequently refer to it.
We also introduce the assumptions that are needed for our construction of  Green's function with Robin boundary condition and for obtaining Gaussian bounds in this section.
In Section~\ref{B3}, we state our main theorems and some corollaries.
Section~\ref{sec:4.2enf} is devoted to discussion of Green's function for elliptic systems with Robin boundary condition.
In Section~\ref{D}, we prove the main theorems stated in Sections~\ref{B3} and \ref{sec:4.2enf}, and in Section~\ref{app}, we provide proofs for some technical lemmas used in the paper.

\section{Preliminaries}		\label{B}

\subsection{Basic notations}
We use $X=(x,t)$ to denote a point in $\bR^{n+1}$; $x=(x_1,\ldots, x_n)$ will always be a point in $\bR^n$.
We also write $Y=(y,s)$, $X_0=(x_0,t_0)$, etc.
We define the parabolic distance between the points $X=(x,t)$ and $Y=(y,s)$ in $\bR^{n+1}$ as
\[
d(X,Y)=\abs{X-Y}_\sP:= \max(\abs{x-y}, \sqrt{\abs{t-s}}),
\]
where $\abs{\,\cdot\,}$ denotes the usual Euclidean norm.
We write $\abs{X}_\sP=\abs{X-0}_\sP$.
For an open set $Q \subset\bR^{n+1}$, we denote
\[
d_X=\dist(X,\partial_p Q)=\inf\bigset{\abs{X-Y}_\sP : Y\in \partial_p Q};\quad \inf \emptyset = \infty,
\]
where $\partial_p Q$ denotes the usual parabolic boundary of $Q$.
We use the following notions for basic cylinders in $\bR^{n+1}$:
\begin{align*}
Q^-_r(X)&=B_r(x) \times (t-r^2,t),\\
Q^+_r(X)&=B_r(x) \times (t,t+r^2),\\
Q_r(X)&= B_r(x)\times (t-r^2,t+r^2),
\end{align*}
where $B_r(x)$ is the usual Euclidean ball of radius $r$ centered at $x\in \bR^n$.
We use the notation
\[
\fint_Q u=\frac{1}{\abs{Q}}\int_Q u.
\]

\subsection{Function spaces}
Throughout the article, we assume that $\Omega$ is a bounded extension domain for $H^1$ functions; i.e. there exists a linear operator $E : W^{1,2}(\Omega) \to W^{1,2}(\bR^n)$ such that
\begin{equation}		\label{extop}
\norm{Eu}_{L^2(\bR^n)} \le \sE_0 \norm{u}_{L^2(\Omega)},\quad
\norm{Eu}_{W^{1,2}(\bR^n)} \le \sE_0 \norm{u}_{W^{1,2}(\Omega)}.
\end{equation}
Such domains include Lipschitz domains, and also locally uniform domains considered by P. Jones; see Rogers \cite{Rogers}.
We identify $H^1(\Omega)= W^{1,2}(\Omega)$ and $H^1_0(\Omega)=W^{1,2}_0(\Omega)$.
We define $H^{1/2}(\partial \Omega)$ as the normed space consisting of all elements of $H^1(\Omega)/ H^1_0(\Omega)$, with the norm
\begin{equation}		\label{2.eq0}
\norm{v}_{H^{1/2}(\partial \Omega)}:=\inf \Set{\norm{u}_{H^1(\Omega)}:u-v\in H^1_0(\Omega)}.
\end{equation}
When $\partial\Omega$ has enough regularity, trace theorems and extension theorems \cite{Tr} readily yield the standard interpretation of $H^{1/2}(\partial\Omega)$.
For a function $u\in H^1(\Omega)$, we let $u\vert_{\partial\Omega} \in H^{1/2}(\partial\Omega)$ to be the equivalence class $u+H^1_0(\Omega)$ and call it the trace of $u$ on $\partial\Omega$.
By abuse of notation, we sometimes write $u$ for $u\vert_{\partial\Omega}$ when there is no danger of confusion.
The spaces $H^{-1}(\Omega)$ and $H^{-1/2}(\partial \Omega)$ denote the Banach spaces consisting of bounded linear functionals on $H^1(\Omega)$ and $H^{1/2}(\partial \Omega)$, respectively.

To avoid confusion, spaces of functions defined on $Q\subset \bR^{n+1}$ will be always written in \emph{script letters} throughout the article.
$\sL_{q,r}(Q)$ is the Banach space consisting of all measurable functions on $Q=\Omega\times (a,b)$ with a finite norm
\[
\norm{u}_{\sL_{q,r}(Q)}=\left(\int_a^b\left(\int_{\Omega} \abs{u(x,t)}^q\, dx\right)^{r/q}dt\right)^{1/r},
\]
where $q\ge 1$ and $r\ge 1$.
$\sL_{q,q}(Q)$ will be denoted by $\sL_q(Q)$.
By $\sC^{\mu,\mu/2}(Q)$ we denote the set of all bounded measurable functions $u$ on $Q$ for which $\abs{u}_{\mu,\mu/2;Q}$ is finite, where
we define the parabolic H\"older norm as follows:
\begin{align*}
\abs{u}_{\mu,\mu/2;Q}&=[u]_{\mu,\mu/2;Q}+\abs{u}_{0;Q}\\
&:= \sup_{\substack{X, Y \in Q\\ X\neq Y}}\frac{\abs{u(X)-u(Y)}}{\abs{X-Y}^\mu_\sP}+\sup_{X\in Q}\;\abs{u(X)}, \quad \mu\in(0,1].
\end{align*}
We write $u\in \sC^\infty_c(Q)$ (resp. $\sC^\infty_c(\bar Q)$) if $u$ is an infinitely differentiable function on $\bR^{n+1}$ with a compact support in $Q$ (resp. $\bar Q$).
We write $D_i u=D_{x_i} u=\partial u/\partial x_i$ ($i=1,\ldots,n$) and $u_t=\partial u /\partial t$.
We also write $Du=D_x u$ for the vector $(D_1 u,\ldots, D_n u)$.
We write $Q(t_0)$ for the set of all points $(x,t_0)$ in $Q$ and $I(Q)$ for
the set of all $t$ such that $Q(t)$ is nonempty.
We denote
\[
\tri{u}_{Q}^2= \int_{Q}  \abs{D_x u}^2 \,dx \,dt+\esssup\limits_{t\in I(Q)} \int_{Q(t)} \abs{u(x,t)}^2\,dx.
\]
The space $\sW^{1,0}_q(Q)$ denotes the Banach space consisting of functions $u\in \sL_q(Q)$ with weak derivatives $D_i u \in \sL_q(Q)$ ($i=1,\ldots,n$) with the norm
\[
\norm{u}_{\sW^{1,0}_q(Q)}=\norm{u}_{\sL_q(Q)}+\norm{D_x u}_{\sL_q(Q)}
\]
and by $\sW^{1,1}_q(Q)$ the Banach space with the norm
\[
\norm{u}_{\sW^{1,1}_q(Q)}=\norm{u}_{\sL_q(Q)}+\norm{D_x u}_{\sL_q(Q)}+ \norm{u_t}_{\sL_q(Q)}.
\]
In the case when $Q$ has a finite height (i.e., $Q\subset \bR^{n}\times(-T,T)$ for some $T<\infty$), we define $\sV_2(Q)$ as the Banach space consisting of all elements of $\sW^{1,0}_2(Q)$ having a finite norm $\norm{u}_{\sV_2(Q)}:=  \tri{u}_{Q}$ and the space $\sV^{1,0}_2(Q)$ is obtained by completing the set $\sW^{1,1}_2(Q)$ in the norm of $\sV_2(Q)$.
When $Q$ has an infinite height, we say that $u \in \sV_2(Q)$ (resp. $\sV^{1,0}_2(Q)$) if $u \in \sV_2(Q_T)$ (resp. $\sV^{1,0}_2(Q_T)$) for all $T>0$, where $Q_T=Q \cap \set{ \abs{t} <T}$, and $\tri{u}_Q <\infty$.
Note that this definition allows that $1\in \sV^{1,0}_2(\Omega\times(-\infty,\infty))$ when $\abs{\Omega}<\infty$.
Finally, we write $u\in \sL_{q,loc}(Q)$ if $u \in \sL_q(Q')$ for all $Q'\Subset Q$ and similarly define $\sW^{1,0}_{q,loc}(Q)$, etc.

\subsection{Robin boundary value problem}
We use the notation $\cB(\cX,\cY)$ for bounded linear operators between two Banach spaces $\cX$ and $\cY$.
We let 
\begin{equation}		\label{eq2.2p}
\Theta(t) \in \cB(H^{1/2}(\partial \Omega)^m, H^{-1/2}(\partial \Omega)^m)\quad \text{for a.e. } t\in \bR
\end{equation}
and assume that
\begin{equation}		\label{eq2.3p}
\esssup_{-\infty<t<\infty}\; \norm{\Theta(t)}_{\cB(H^{1/2}(\partial \Omega)^m, H^{-1/2}(\partial \Omega)^m)} <\infty.
\end{equation}
For $\vec u, \vec v \in H^1(\Omega)^m$, we write 
\[
\ip{\Theta(t)\vec u, \vec v} := \left[\Theta(t) (\vec u\vert_{\partial\Omega})\right](\vec v\vert_{\partial\Omega}).
\]
The following is then an immediate consequence of \eqref{2.eq0}, \eqref{eq2.2p}, and \eqref{eq2.3p}:
\[
\abs{\ip{\Theta(t) \vec u, \vec v}} \le C \norm{\vec u}_{H^1(\Omega)} \norm{\vec v}_{H^1(\Omega)},\quad \forall \vec u, \vec v \in H^1(\Omega)^m
\]
for a.e. $t\in \bR$.
The adjoint operator $\Theta^{*}(t) \in  \cB(H^{1/2}(\partial \Omega)^m, H^{-1/2}(\partial \Omega)^m)$ is defined by the usual relation $\ip{\Theta^{*}(t)\vec u,\vec v}=\ip{\Theta(t) \vec v,\vec u}$.
Let $Q=\Omega\times (a,b)$ and $S =\partial\Omega\times (a,b)$,  where $-\infty \le a <b \le \infty$.
We say that $\vec u$ is a weak solution of
\begin{equation}		\label{RP}
\tag{RP}
\sL \vec u = \vec f\;\text{ in }\; Q,\quad \partial \vec u/\partial  \nu+\Theta \vec u =0\;\text{ on }\;S
\end{equation}
if  $\vec u \in \sV_2(Q)^m$ and satisfies
\[
-\int_Q \vec u\cdot \vec \phi_t\,dX+\int_Q A^{\alpha\beta}D_\beta \vec u\cdot D_\alpha \vec \phi\,dX+\int_a^b \ip{\Theta\vec u,\vec \phi}\,dt
=\int_Q \vec f\cdot \vec \phi \,dX
\]
for $\vec \phi\in \sC^\infty(\bar Q)^m$ that vanishes for $t=a$ and $t=b$.
Similarly, we say that $\vec u$ is a weak solution of
\begin{equation}		\label{RPs}
\tag{RP*}
\sL^{*} \vec u = \vec f\;\text{ in }\; Q,\quad \partial \vec u/\partial  \nu^{*}+\Theta^{*} \vec u =0\;\text{ on }\;S
\end{equation}
if  $\vec u \in \sV_2(Q)^m$ and satisfies
\[
\int_Q \vec u\cdot \vec \phi_t\,dX+\int_Q A^{\alpha\beta}D_\beta \vec \phi \cdot D_\alpha \vec u \,dX+\int_a^b \ip{\Theta\vec \phi,\vec u}\,dt
=\int_Q \vec f\cdot \vec \phi \,dX
\]
for $\vec \phi\in \sC^\infty(\bar Q)^m$ that vanishes for $t=a$ and $t=b$.

For $a\neq -\infty$, we say that $\vec u$ is a weak solution of the problem
\begin{equation}			\label{B2.eq2}
\left\{
\begin{aligned}
\sL\vec u=\vec f &\quad \text{in }\, Q,\\
\partial \vec u/\partial \nu + \Theta \vec u =0 &\quad \text{on }\, S,\\
\vec u=\vec \psi_0 &\quad \text{on }\, \Omega\times \set{a}, 
\end{aligned}
\right.
\end{equation}
if $\vec u\in \sV_2^{1,0}(Q)^m$ and satisfies for all $t_1 \in [a,b]$ the identity 
\begin{multline}		\label{eq2.07ej}
\int_\Omega \vec u(\cdot, t_1)\cdot \vec \phi(\cdot,t_1)\,dx-\int_a^{t_1}\!\!\!\int_\Omega \vec u\cdot \vec \phi_t\,dX+\int_a^{t_1}\!\!\!\int_\Omega A^{\alpha\beta}D_\beta \vec u\cdot D_\alpha \vec \phi\,dX+\int_a^{t_1} \ip{\Theta\vec u,\vec \phi}\,dt\\
=\int_a^{t_1}\!\!\!\int_\Omega \vec f\cdot \vec \phi \,dX +\int_\Omega \vec \psi_0\cdot \vec \phi(\cdot,a)\,dx,\quad \forall \vec \phi \in \sC^\infty(\bar Q)^m. 
\end{multline}
Similarly, we say that $\vec u$ is a weak solution of the (backward) problem
\begin{equation}			\label{B2.eq3}
\left\{
\begin{aligned}
\sL^*\vec u=\vec f &\quad \text{in }\, Q,\\
\partial \vec u/\partial \nu^{*} + \Theta^{*} \vec u=0 &\quad \text{on }\, S,\\
\vec u=\vec \psi_0 &\quad \text{on }\, \Omega\times \set{b}, 
\end{aligned}
\right.
\end{equation}
if $\vec u\in \sV_2^{1,0}(Q)^m$ and satisfies for all $t_1 \in [a,b]$ the identity 
\begin{multline*}
\int_\Omega \vec u(\cdot, t_1)\cdot \vec \phi(\cdot,t_1)\,dx+\int_{t_1}^b\!\!\!\int_\Omega \vec u\cdot \vec \phi_t\,dX+\int_{t_1}^b\!\!\!\int_\Omega A^{\alpha\beta}D_\beta \vec \phi\cdot D_\alpha \vec u\,dX+\int_{t_1}^b \ip{\Theta \vec \phi,\vec u}\,dt\\
=\int_{t_1}^b\!\!\!\int_\Omega \vec f\cdot \vec \phi \,dX +\int_\Omega \vec \psi_0\cdot \vec \phi(\cdot,b)\,dx,\quad \forall \vec \phi \in \sC^\infty(\bar Q)^m. 
\end{multline*}

\subsection{Robin Green's function}		\label{2.rp}
Let $Q:=\Omega\times (-\infty,\infty)$.
We say that an $m\times m$ matrix valued function $\vec  \cG(X,Y)=\vec  \cG(x,t,y,s)$, with entries $\cG_{ij}: Q\times Q\mapsto [-\infty,\infty]$, is a Green's function for Robin problem \eqref{RP} if it satisfies the following properties.
\begin{enumerate}[a)]
\item
For all $Y\in Q$, we have $\vec  \cG(\cdot,Y)\in \sW^{1,0}_{1,loc}(Q)^{m\times m}$ and $\vec \cG(\cdot,Y) \in \sV_2(Q\setminus Q_r(Y))^{m\times m}$ for any $r>0$. 

\item
For all $Y\in Q$,  we have $\sL \vec \cG(\cdot,Y)=\delta_Y \vec I$ in $Q$ and $\partial\vec \cG(\cdot, Y)/\partial \nu + \Theta \vec \cG(\cdot, Y)=0$ on $S=\partial Q$
in the sense that for any $\vec\phi\in \sC^\infty_c(\bar Q)^m$, we have
\[
-\int_Q \vec \cG_{\cdot k}(\cdot,Y) \cdot \vec \phi_t\,dX+\int_Q A^{\alpha\beta} D_\beta \vec \cG_{\cdot k}(\cdot,Y)\cdot D_\alpha \vec \phi\,dX+\int_{-\infty}^\infty \ip{\Theta \vec \cG_{\cdot k}(
\cdot,Y)\vert_{\partial\Omega},\vec \phi\vert_{\partial\Omega}}\,dt=\phi^k(Y),
\]
where $\vec \cG_{\cdot k}(X, Y)$ is the $k$-th ($k=1,\ldots, m$) column of $\vec \cG(X,Y)$.
\item
For any $\vec f\in \sC^\infty_c(\bar Q)^m$, the function $\vec u$ given by
\[
\vec u(X):=\int_Q \vec \cG(Y,X)^\top\vec f(Y)\,dY
\]
is a weak solution of \eqref{RPs}.
\end{enumerate}
We note that part c) of the above definition gives uniqueness of Robin Green's function.

\subsection{Basic assumptions}
We make the following assumptions (H1) and (H2) to construct the Robin Green's function in $Q = \Omega\times (-\infty,\infty)$.
\begin{enumerate}[{\bf{H}1.}]
\item
We assume that $\Omega$ is an extension domain for $H^1$ function so that \eqref{extop} holds for some constant $\sE_0$.
We assume that $\Theta$ satisfies \eqref{eq2.2p}, \eqref{eq2.3p} and there exist constants $\tilde{\lambda} \in (0,\lambda)$ and $\vartheta_0>0$ such that for all $\vec u \in H^1(\Omega)^m$, we have
\begin{equation}		\label{2.eq4}
\vartheta_0 \norm{\vec u}_{H^1(\Omega)}^2  \le  \tilde{\lambda} \norm{D\vec u}_{L^2(\Omega)}^2+\ip{\Theta(t) \vec u,\vec u}
\end{equation}
for a.e. $t \in (-\infty,\infty)$.
\item
There exist constants $\mu_0\in(0,1]$ and $A_0>0$ such that if $\vec u$ is a weak solution of $\sL\vec u= 0$ (resp. $\sL^{*} \vec u=0$) in $\tilde{Q}=Q^{-}_R(X)$ (resp. $\tilde{Q}=Q^{+}_R(X)$),  where $X\in Q$ and $0<R< \dist(X,\partial_p Q)$, then we have
\[
[\vec u]_{\mu_0,\mu_0/2; \frac{1}{2}\tilde{Q}}\le A_0R^{-\mu_0}\left(\fint_{\tilde{Q}}\,\abs{\vec u(Y)}^2\,dY \right)^{1/2},
\]
where $\frac{1}{2}\tilde{Q}=Q^{-}_{R/2}(X)$ $\big($resp. $\frac{1}{2}\tilde{Q}=Q^{+}_{R/2}(X)\big)$.

\end{enumerate}

The following assumption (H3) is used to obtain global Gaussian estimates for the Robin Green's function.
We point out that the integral appearing in (H3) is different from those in the condition (A3) of \cite{CK2} and the condition (LB) of \cite{CDK2}.
\begin{enumerate}[{\bf{H}1.}]
\setcounter{enumi}{2}

\item
For any $\vec u, \vec v \in H^1(\Omega)^m$ satisfying $\vec u \cdot \vec v \ge 0$ a.e.  in $\Omega$, we have
\begin{equation}		\label{eq2.09local}
\ip{\Theta(t)\vec u, \vec v} \ge 0\quad \text{for a.e. }t \in(-\infty,\infty).
\end{equation}
Also, there exist constants $A_1, R_1>0$ such that if $\vec u$ is a weak solution of
\begin{align*} 
\sL\vec u=0 \quad \text{in }\, \Omega\times(a,b), &\qquad \partial \vec u/\partial \nu +\Theta \vec u =0 \quad \text{on }\, \partial\Omega\times(a,b),\\
\big ( \text{resp.} \quad \sL^*\vec u=0 \quad \text{in }\, \Omega\times(a,b),&\qquad \partial \vec u/\partial \nu^{*} +\Theta^{*} \vec u=0 \quad \text{on }\, \partial\Omega\times(a,b), \big )
\end{align*}
then for a.e. $x \in \overline \Omega$, we have
\begin{align*}
\abs{\vec u(x,b)} &\le A_1 R^{-(n+2)/2}\norm{\vec u}_{\sL_2(\Omega \times (b-R^2,b))},\\
\big ( \text{resp.} \quad \abs{\vec u(x,a)} &\le A_1 R^{-(n+2)/2} \norm{\vec u}_{\sL_2(\Omega \times (a, a+R^2))}, \big )
\end{align*}
where $R=\min(\sqrt{b-a}, R_1)$.
\end{enumerate} 

\begin{remark}		\label{rmk2.1pe}
We note that if $\Theta \in \cB(H^{1/2}(\partial \Omega)^m, H^{-1/2}(\partial \Omega)^m)$ satisfies
\begin{equation}		\label{2.eq4-1}
\vartheta_0 \norm{\vec u}_{H^1(\Omega)}^2  \le  \tilde{\lambda} \norm{D\vec u}_{L^2(\Omega)}^2+\ip{\Theta \vec u,\vec u},
\end{equation}
then $\norm{\cdot}_{\Theta;\Omega}$ defined by
\[
\norm{\vec u}_{\Theta;\Omega}:=\left(
 \tilde{\lambda} \norm{D \vec u}_{L^2(\Omega)}^2 + \ip{\Theta \vec u, \vec u}\right)^{1/2}
\]
gives an equivalent norm in $H^1(\Omega)^m$. 
In fact, 
\[
(\vec u, \vec v)_{\Theta}:=\tilde{\lambda} \int_\Omega \sum_{\alpha=1}^n D_\alpha \vec u \cdot D_\alpha \vec v \,dx + \frac{1}{2}\left(\ip{\Theta \vec u, \vec v}+\ip{\Theta \vec u, \vec v}\right)
\]
is an equivalent inner product on $H^1(\Omega)^m$.
\end{remark}

\begin{remark}		
Suppose $\Theta \in \cB(H^{1/2}(\partial \Omega)^m, H^{-1/2}(\partial \Omega)^m)$ has the following properties:
\begin{enumerate}[i)]
\item
$\ip{\Theta \vec u, \vec u}\ge 0$ for all $\vec u \in H^1(\Omega)^m$.
\item 
$\ip{\Theta \vec \xi, \vec \xi}=0$ for $\vec \xi \in \bR^m$ implies $\vec \xi=0$.
\end{enumerate}
Then, one can obtain \eqref{2.eq4-1} by a usual contradiction argument based on Rellich-Kondrachov compactness theorem (cf. proof of Lemma~\ref{lem2.2pe}).
\end{remark}

\begin{remark}		
Suppose $\Theta=\Theta^{(1)}+\Theta^{(2)}$, where $\Theta^{(1)}$ satisfies \eqref{2.eq4-1} and $\ip{\Theta^{(2)} \vec u, \vec u}\ge 0$ for all $\vec u \in H^1(\Omega)^m$.
Then, $\Theta$ satisfies \eqref{2.eq4-1} as well.
\end{remark}

\begin{remark}		\label{rmk2.4pe}
Below are some examples of cases when the condition (H2) holds.
\begin{enumerate}[i)]
\item
The scalar case ($m=1$) is a consequence of De Giorgi-Moser-Nash theory.

\item
$n=2$ and the coefficients of $\sL$ are time-independent (see \cite[Theorem~3.3]{Kim}).

\item
The coefficients of $\sL$ belong to $VMO_x$ (see \cite[Lemma~2.3]{CDK});
we say that $f$ belongs to $VMO_x$ if $\lim_{\rho\to 0} \omega_\rho(f)=0$, where
\[
\omega_\rho(f)=\sup_{X\in\bR^{n+1}}\sup_{r \leq \rho} \fint_{t-r^2}^{t+r^2}\!\!\!\fint_{B_r(x)} \Abs{f(y,s)-\fint_{B_r(x)}f(\cdot,s)}\,dy\,ds.
\]
\end{enumerate}
\end{remark}

\begin{remark}		\label{rmk2.5pe}
Note that \eqref{eq2.09local} in (H3) requires that $\Theta(t) \ge 0$ for a.e. $t\in \bR$; i.e.,  for a.e. $t\in \bR$, we have $\ip{\Theta(t) \vec u, \vec u} \ge 0$ for all $\vec u \in H^1(\Omega)$.
It should be clear that \eqref{eq2.09local} is satisfied if $\Theta$ is defined by
\[
\ip{\Theta(t)\vec u, \vec v} := \int_{\partial\Omega}  \theta(x,t) \vec u(x) \cdot \vec v(x) \,dS_x,\quad \forall \vec u, \vec v \in H^{1/2}(\partial\Omega)^m,
\]
where $\theta$ is a (symmetric) nonnegative definite $m\times m$ matrix-valued function.
Also, in (H3) we may take $R_1=1$ or $R_1=\diam \Omega$ (because we assume $\Omega$ is bounded), possibly at the cost of increasing the constant $A_1$.
\end{remark}
\subsection{Auxiliary results}
We recall that the following multiplicative inequality holds for any $u\in W^{1,2}(\bR^n)$ with $n\ge 1$ (see \cite[Theorem 2.2]{LSU}):
\begin{equation}		\label{eq2.10ar}
\norm{u}_{L^{2(n+2)/n}}\le C(n)\norm{Du}_{L^2}^{n/(n+2)}\norm{u}_{L^2}^{2/(n+2)}.
\end{equation}
If we assume (H1), then by \eqref{eq2.10ar},  \eqref{extop}, and Remark~\ref{rmk2.1pe}, for any $\vec u\in H^1(\Omega)^m$ and $-\infty<t< \infty$, we have
\begin{multline*}
\norm{\vec u}_{L^{2(n+2)/n}(\Omega)}
\le C \norm{D(E \vec u)}_{L^2(\bR^n)}^{n/(n+2)} \norm{E \vec u}_{L^2(\bR^n)}^{2/(n+2)} \le C \norm{\vec u}_{H^1(\Omega)}^{n/(n+2)} \norm{\vec u}_{L^2(\Omega)}^{2/(n+2)}\\
\le C\left(\tilde{\lambda} \norm{D\vec u}_{L^2(\Omega)}^2+\ip{\Theta(t)\vec u, \vec u}\right)^{n/(2(n+2))}\norm{\vec u}_{L^2(\Omega)}^{2/(n+2)}. 
\end{multline*}
Therefore, there exists a constant $\gamma_{\Theta}$ such that 
\begin{equation}
\norm{\vec u}_{L^{2(n+2)/n}(\Omega)}
\le \gamma_\Theta \left(\tilde{\lambda} \norm{D\vec u}_{L^2(\Omega)}^2+\ip{\Theta(t)\vec u, \vec u}\right)^{n/(2(n+2))}\norm{\vec u}_{L^2(\Omega)}^{2/(n+2)},\quad \forall \vec u\in H^1(\Omega)^m.
\label{eq2.11ar}
\end{equation}
Let $Q=\Omega\times (a,b)$ with $-\infty \le a <b \le \infty$.
We define
\begin{equation}	\label{vtq}
\tri{\vec u}_{\Theta; Q}:=\left( \esssup\limits_{a<t<b} \int_{\Omega} \abs{\vec u(x,t)}^2\,dx+\tilde{\lambda} \int_Q  \abs{D \vec u}^2 \,dx\,dt+ \int_a^b \ip{\Theta \vec u,\vec u}\,dt \right)^{1/2}.
\end{equation}
We note that by Remark~\ref{rmk2.1pe}, we have
\begin{equation}		\label{eq2.19ar}
\tri{\vec u}_Q =\norm{\vec u}_{\sV_2(Q)} \le C(\vartheta_0) \tri{\vec u}_{\Theta;Q}.
\end{equation}
The membership of $\vec u$ in  $\sV_2(Q)^m$ implies that $\vec u(\cdot, t) \in H^1(\Omega)^m$ for a.e. $t \in (a,b)$, and thus, we derive from \eqref{eq2.11ar} that
\begin{multline*}
\int_a^b\!\!\!\int_\Omega \abs{\vec u}^{2(n+2)/n}\,dx\,dt
\le \gamma_\Theta^{2(n+2)/n} \left(\esssup_{a<t<b} \int_\Omega \abs{\vec u(x,t)}^2\,dx \right)^{2/n}\times\\
\left(\tilde{\lambda} \int_a^b\!\!\!\int_\Omega \abs{D\vec u}^2\,dx\,dt+ \int_a^b \ip{\Theta(t)\vec u, \vec u}\,dt\right).
\end{multline*}
Therefore, by the definition \eqref{vtq}, we obtain
\begin{equation}		\label{eq2.14ar}
\norm{\vec u}_{\sL_{2(n+2)/n}(Q)}\le \gamma_{\Theta} \tri{\vec u}_{\Theta; Q}, \quad \forall \vec u\in \sV_2(Q)^m.
\end{equation}
We also recall that for any $u\in \sV_2(\Omega\times (a,b))$, we have
\begin{align}		\label{2.eq4d}
\norm{u}_{\sL_{2(n+2)/n}(\Omega\times (a,b))} &\le \left(2\beta +\left\{(b-a)^{\frac{n}{2}}\abs{\Omega}^{-1}\right\}^{\frac{1}{n+2}}\right)\tri{u}_{\Omega\times (a,b)},\\
\label{eq2.16wine}
\norm{u}_{\sL_{2(n+1)/n}(\partial\Omega\times (a,b))} &\le \beta\left(1+\left\{(b-a)^{\frac{n}{2}}\abs{\Omega}^{-1}\right\}^{\frac{1}{n+1}}\right) \tri{u}_{\Omega\times (a,b)},
\end{align}
where $\beta=\beta(n, \Omega, \partial\Omega)$; see \cite[pp. 77 -- 78]{LSU}.

Finally, we state some useful lemmas whose proofs will be given in Section~\ref{app}.
\begin{lemma}		\label{B3.lem1}
Assume  (H1) and let $\vec \psi_0 \in L^2(\Omega)^m$ and $\vec f \in \sL_{2,1}(Q)^m$, where $Q=\Omega\times (a,b)$ and $-\infty<a<b<\infty$.
Then, there exists a unique weak solution $\vec u$ in $\sV^{1,0}_2(Q)^m$ of the problem \eqref{B2.eq2}. 
If $\norm{\vec f}_{\sL_{2(n+2)/(n+4)}(Q)}<\infty$, then the weak solution $\vec u$ of the problem \eqref{B2.eq2} satisfies an energy inequality
\begin{equation}			\label{enieq}
\tri{\vec u}_{\Theta;Q}  \le C \left\{ \norm{\vec f}_{\sL_{2(n+2)/(n+4)}(Q)}+ \norm{\vec\psi_0}_{L^2(\Omega)} \right\},
\end{equation}
where $C$ depends only on $n, m, \lambda$ and parameters in (H1).
A similar statement is true for the problem \eqref{B2.eq3}.
\end{lemma}

\begin{lemma}		\label{lem2.2pe}
Let $\Omega$ be a Lipschitz domain and $Q=\Omega\times (-\infty,\infty)$.
Suppose 
\begin{equation}	\label{eq2.09ex}
\theta \in \sL_{p,\infty}(\partial Q)^{m\times m}, \;\text{ where  $p=n-1$ if  $n\ge 3$ and $p\in (1,\infty]$ if $n=2$}.
\end{equation}
Let $\Theta=M_\theta$, the operator of multiplication with $\theta$ defined by
\[
\ip{\Theta(t)\vec u, \vec v} := \int_{\partial\Omega}  \theta(x,t) \vec u(x) \cdot \vec v(x) \,dS_x,\quad \forall \vec u, \vec v \in H^{1/2}(\partial\Omega)^m.
\]
Suppose  that there is $\delta>0$ such that for a.e. $t$, we have
\begin{equation}		\label{eq2.10ex}
\inf_{\vec e \in \bR^m,\; \abs{\vec e}=1} \left(\int_{\partial\Omega} \theta(x,t) \,dS_x \right) \vec e \cdot \vec e\ge \delta> 0.
\end{equation}
If either $\theta \vec \xi \cdot \vec \xi \ge 0$ for all $\vec \xi\in \bR^m$ or $\norm{\theta}_{\sL_{p,\infty}}$ is sufficiently small, then (H1) holds.
\end{lemma}

\section{Main theorems}		\label{B3}

\begin{theorem}		\label{B3.thm1}
Assume the conditions (H1) and (H2).
Let $Q=\Omega \times (-\infty,\infty)$.
Then there exists a unique Green's function $\vec \cG(X,Y)=\vec \cG(x,t,y,s)$ for Robin problem \eqref{RP}.
It is continuous in $\set{(X,Y)\in Q\times Q:X\neq Y}$ and vanishes for $t<s$.
We have $\vec \cG^{*}(X,Y)= \vec \cG(Y,X)^\top$ is a Green's function for the adjoint problem \eqref{RPs}.
For any $\vec f =(f^1,\ldots, f^m)^\top \in \sC^\infty_c(\bar Q)^m$, the function $\vec u$ given by
\begin{equation}			\label{eq3.04c}
\vec u(X):= \int_{Q} \vec \cG(X,Y) \vec f(Y)\,dY
\end{equation}
is a weak solution in $\sV^{1,0}_2(Q)^m$ of
\begin{equation}			\label{eq3.05a}
\sL \vec u = \vec f\;\text{ in }\;Q,\quad \partial \vec u/\partial \nu+\Theta \vec u =0\;\text{ on }\;\partial_p Q.
\end{equation}
Moreover, for all $\vec \psi =(\psi^1,\ldots,\psi^m)^\top \in L^2(\Omega)^m$, the function given by
\begin{equation}		\label{EP.1b}
\vec u(x,t)=\int_\Omega \vec \cG (x,t,y,s) \vec \psi(y)\,dy
\end{equation}
is a unique weak solution in $\sV^{1,0}_2(\Omega\times(s,\infty))^m$ of the problem
\begin{equation}		\label{EP.1b1}			
\left\{
\begin{aligned}
\sL\vec u= 0\quad &\text{in }\; \Omega\times(s,\infty)\\
\partial \vec u/\partial \nu+\Theta \vec u= 0 \quad&\text{on }\; \partial\Omega\times(s,\infty)\\
\vec u(\cdot, s)=\vec \psi \quad&\text{on }\; \Omega.
\end{aligned}
\right.
\end{equation}
Furthermore, for $X, Y \in Q$ satisfying $0<\abs{X-Y}_{\sP} < \frac{1}{2} \dist(Y, \partial_p Q)$, we have
\begin{equation}		\label{eq3.07c}
\abs{\vec \cG(X,Y)} \le C \abs{X-Y}_{\sP}^{-n},
\end{equation}
and for $X, X', Y\in Q$ satisfying $2\abs{X-X'}_{\sP}<\abs{X-Y}_{\sP}<\frac{1}{2} \dist(Y, \partial_p Q)$, 
\begin{equation}		\label{eq3.08d}
\abs{\vec \cG(X,Y)-\vec \cG(X',Y)}\le C\abs{X-X'}_{\sP}^{\mu_0} \abs{X-Y}_{\sP}^{-n-\mu_0},
\end{equation}
where the constant $C$ depend only on $n, m, \lambda$ and parameters in (H1), (H2).
Finally, $\vec \cG^{*}(X,Y)= \vec \cG(Y,X)^\top$ is Green's function for the adjoint problem \eqref{RPs}.
\end{theorem}

\begin{remark}		\label{rmk3.8}
It will be clear from the proof that  for any $Y\in Q$ and $0<r<d_Y=\dist(Y,\partial_p Q)$, we have
\begin{enumerate}[i)]
\item
$\tri{(1-\eta)\vec \cG_{\cdot k}(\cdot,Y)}_{Q}\le C r^{-n/2}$.
\item
$\norm{\vec \cG(\cdot,Y)}_{\sL_{2(n+2)/n}( Q\setminus \bar Q_r(Y))}\le C r^{-n/2}$.
\item
$\Abs{\set{X\in Q:\abs{\vec \cG(X,Y)}>\tau}}\le C\tau^{-\frac{n+2}{n}}, \quad \forall \tau>d_Y^{-n}$.
\item
$\Abs{\set{X\in Q:\abs{D_x\vec \cG(X,Y)}>\tau}}\le C\tau^{-\frac{n+2}{n+1}}, \quad \forall \tau>d_Y^{-(n+1)}$.
\item
$\norm{\vec \cG(\cdot,Y)}_{\sL_p(Q_r(Y))}\le C r^{-n+\frac{n+2}{p}} \quad \text{for }\; p\in [1,\frac{n+2}{n})$.
\item
$\norm{D\vec \cG(\cdot,Y)}_{\sL_p(Q_r(Y))}\le C r^{-n-1+\frac{n+2}{p}} \quad \text{for }\; p\in [1,\frac{n+2}{n+1})$.
\end{enumerate}
\end{remark}

\begin{theorem}		\label{B3.thm2}
Assume the conditions (H1) - (H3).
Then, for $t>s$, we have Gaussian bound for the Robin Green's function
\begin{equation}		\label{3.eq2}
\abs{\vec \cG(x,t,y,s)}\le \frac{C}{\min\left\{\sqrt{t-s}, \diam \Omega\right\}^{n}}\exp\left(\frac{-\kappa\abs{x-y}^2}{t-s}\right),
\end{equation}
where $\kappa=\kappa(\lambda)>0$ and $C=C(n,m,\lambda, A_1)$.
\end{theorem}

\begin{theorem}	\label{thm3.3flu}
Assume the conditions (H1) and let  $\Theta=M_\theta$, the operator of multiplication with $\theta$, where $\theta$ be an $m\times m$ matrix-valued function defined on $\partial \Omega\times (-\infty,\infty)$.
Assume that one of the following conditions holds:
\begin{enumerate}[(i)]
\item
$m=1$ (i.e. the scalar case), $\Omega$ is a Lipschitz domain, and $\theta \ge 0$.

\item
$\Omega$ is a $C^1$ domain, $A^{\alpha\beta}$ are in $VMO_x$ (see Remark~\ref{rmk2.4pe}), and $\theta$ is bounded.

\item
$n=2$, $\Omega$ is a Lipschitz domain, $A^{\alpha\beta}(X)=A^{\alpha\beta}(x)$, $\theta(X)=\theta(x)$ (i.e. $t$-independent), and  $\theta$ is bounded.
\end{enumerate}
Then, there exist $C>0$ and $r_0>0$ such that if $\vec u$ is a weak solution of 
\begin{equation}		\label{eqj.1b}
\left\{
\begin{aligned}
\vec u_t-D_\alpha(A^{\alpha\beta}D_\beta \vec u)=0 &\quad \text{in }\, Q:=\Omega\times (a,b),\\
\partial \vec u/\partial \nu+\theta\vec u=0 &\quad \text{on }\, S:=\partial \Omega\times (a,b),
\end{aligned}
\right.
\end{equation}
we have for any $x_0\in \overline \Omega$ and $0<R<\min(\sqrt{b-a},r_0)$ that
\begin{equation}		\label{30eq2a}
\norm{\vec u}_{\sL_\infty(Q^{-}_{R/2}(X_0)\cap Q)} \le C R^{-(n+2)/2}\norm{\vec u}_{\sL_2(\Omega \times (b-R^2,b))};\quad X_0:=(x_0, b).
\end{equation}
Analogous statement is true for the corresponding adjoint case.
\end{theorem}

The following corollaries are then immediate consequences of the above theorems, Lemma~\ref{lem2.2pe}, and Remarks~\ref{rmk2.4pe} and \ref{rmk2.5pe}.

\begin{corollary}
Let $m=1$ and $\Omega$ be a Lipschitz domain.
Suppose $\Theta=M_\theta$, where  $\theta \ge 0$ and satisfies \eqref{eq2.09ex} and \eqref{eq2.10ex}.
Then there exists a (scalar) Green's function for \eqref{RP} and it satisfies the conclusions of Theorems \ref{B3.thm1} and \ref{B3.thm2}.
\end{corollary}

\begin{corollary}
Let $\Omega$ be a $C^1$ domain and the coefficients of $\sL$ belong to $VMO_x$.
Suppose $\Theta=M_\theta$, where $\theta \in \sL_\infty(\partial\Omega\times (-\infty,\infty))^{m\times m}$ is such that $\theta$ is nonnegative definite and satisfies \eqref{eq2.10ex}.
Then there exists a Green's function for \eqref{RP} and it satisfies the conclusions of Theorems \ref{B3.thm1} and \ref{B3.thm2}.
\end{corollary}

\begin{corollary}
Let $n=2$ and $\Omega$ be a Lipschitz domain and the coefficient of $\sL$ are $t$-independent.
Suppose $\Theta=M_\theta$, where $\theta \in \sL_\infty(\partial\Omega\times (-\infty,\infty))^{m\times m}$ is such that $\theta$ is $t$-independent, nonnegative definite and satisfies \eqref{eq2.10ex}.
Then there exists a Green's function for \eqref{RP} and it satisfies the conclusions of Theorems \ref{B3.thm1} and \ref{B3.thm2}.
\end{corollary}

\section{Elliptic Robin Green's function}			\label{sec:4.2enf}
Let us consider elliptic differential operator of the form
\[
L \vec u=  -\frac{\partial}{\partial x_\alpha} \left(A^{\alpha\beta}(x) \frac{\partial \vec u}{\partial x_\beta}\right),
\]
where $A^{\alpha\beta}(x)$ are $m\times m$ matrices whose elements $a^{\alpha\beta}_{ij}(x)$ are bounded measurable functions satisfying \eqref{B2.eq1}, and its adjoint operator $L^{*}$ defined by
\[
L^{*} \vec u=  -\frac{\partial}{\partial x_\alpha} \left(A^{\beta\alpha}(x)^\top \frac{\partial \vec u}{\partial x_\beta}\right).
\]
We consider Robin boundary value problem
\begin{equation}	\label{ERP}
\tag{RP'}
\left\{
\begin{aligned}
L\vec u&=\vec f \quad\text{in }\;\Omega,\\
\partial \vec u/\partial \nu+\Theta \vec u&= 0\quad \text{on }\;\partial\Omega,
\end{aligned}
\right.
\end{equation}
where $\Theta \in \cB(H^{1/2}(\partial \Omega)^m, H^{-1/2}(\partial \Omega)^m)$.
Given $\vec f \in H^{-1}(\Omega)^m$, we say that $\vec u$ is a weak solution of the problem \eqref{ERP} if we have $\vec u \in H^1(\Omega)^n$ and
\[
\int_\Omega A^{\alpha\beta}D_\beta \vec u\cdot D_\alpha \vec \phi\,dx+ \ip{\Theta \vec u,\vec \phi}= \ip{\vec f,\vec \phi}, \quad \forall \vec \phi \in H^{1}(\Omega)^m.
\]
It can be shown via standard elliptic theory that there exists a unique weak solution in $H^{1}(\Omega)^m$ of the problem \eqref{ERP}.
We say that an $m\times m$ matrix valued function $\vec G(x,y)$ is the Green's function for \eqref{ERP}
if it satisfies the following properties:
\begin{enumerate}[i)]
\item
$\vec G(\cdot,y) \in W^{1,1}_{loc}(\Omega)^{m\times m}$ and $\vec G(\cdot,y) \in W^{1,2}(\Omega\setminus B_r(y))^{m\times m}$ for all $y\in\Omega$ and $r>0$.
\item
$L\vec G(\cdot,y)=\delta_y \vec I$ in $\Omega$,  $[\partial \vec /\partial \nu + \Theta ]  \vec G(\cdot,y)=0$ on $\partial\Omega$ for all $y\in\Omega$ in the  sense that
\[
\int_\Omega A^{\alpha\beta} D_\beta \vec G_{\cdot k}(\cdot,y)\cdot D_\alpha \vec \phi\,dx+ \ip{\Theta \vec G_{\cdot k}(\cdot,y)\vert_{\partial\Omega},\vec \phi}=\phi^k(y),\quad \forall \vec \phi \in C^\infty(\bar\Omega)^m.
\]
\item
For any $\vec f \in C^\infty(\bar\Omega)^m$, the function $\vec u$ defined by
\begin{equation}	\label{eq4.01e}
\vec u(x)=\int_\Omega \vec G(y,x)^\top \vec f(y)\,dy
\end{equation}
is the weak solution in $H^1(\Omega)^m$ of the adjoint Robin problem
\begin{equation}		\label{eq4.02e}
\left\{
\begin{aligned}
L^{*}\vec u&=\vec f \quad\text{in }\;\Omega,\\
\partial \vec u/\partial \nu^{*}+\Theta^{*} \vec u&= 0\quad \text{on }\;\partial\Omega.
\end{aligned}
\right.
\end{equation}
\end{enumerate}
Note the property iii) gives the uniqueness of the Robin Green's function.

The following assumptions are parallel with the conditions (H1) - (H3).
\begin{enumerate}[{\bf{H}1'.}]
\item
We assume that $\Omega$ is an extension domain for $H^1$ function so that \eqref{extop} holds for some constant $\sE_0$.
We assume that $\Theta \in \cB(H^{1/2}(\partial \Omega)^m, H^{-1/2}(\partial \Omega)^m)$ and there exist constants $\tilde{\lambda} \in (0,\lambda)$ and $\vartheta_0>0$ such that
\[
\vartheta_0 \norm{\vec u}_{H^1(\Omega)}^2  \le  \tilde{\lambda} \norm{D\vec u}_{L^2(\Omega)}^2+\ip{\Theta \vec u,\vec u},\quad \forall  \vec u \in H^1(\Omega)^m.
\]

\item
There exist constants $\mu_0\in(0,1]$, and $A_0>0$ such that if $\vec u$ is a weak solution of either $L\vec u= 0$ or $L^{*} \vec u=0$ in $B_r=B_r(x)$,  where $x\in \Omega$ and $0<r< \dist(x,\partial \Omega)$, then we have
\[
[\vec u]_{\mu_0; B_{r/2}} \le A_0 r^{-\mu_0}\left(\fint_{B_r} \abs{\vec u}^2\right)^{1/2},
\]
where $[\vec u]_{\mu_0; B_{r/2}}$ is the usual H\"older seminorm.

\item
We have $\ip{\Theta \vec u, \vec v} \ge 0$ for all $\vec u, \vec v \in H^1(\Omega)^m$ satisfying $\vec u\cdot \vec v \ge 0$ a.e. in $\Omega$.
Also, there exist constants $A_1, R_1>0$ such that if $\vec u$ is a weak solution of
\begin{align*}
\vec u_t - L\vec u=0 \quad \text{in }\, \Omega\times(0,T), &\qquad \partial \vec u/\partial \nu +\Theta \vec u =0 \quad \text{on }\, \partial\Omega\times(0,T)\\
\text{or} \quad \vec u_t-L^{*} \vec u=0 \quad \text{in }\, \Omega\times(0,T),&\qquad \partial \vec u/\partial \nu^{*} +\Theta^{*} \vec u=0 \quad \text{on }\, \partial\Omega\times(0,T),
\end{align*}
then $\vec u$ is locally bounded and for a.e. $x \in \bar\Omega$, we have
\[
\abs{\vec u(x,T)} \le A_1 R^{-(n+2)/2}\norm{\vec u}_{\sL_2(\Omega \times (T-R^2,T))},
\]
where $R = \min(\sqrt{T}, R_1)$.
\end{enumerate}

\begin{remark}		\label{rmk4.4pe}
Similar to Remark~\ref{rmk2.4pe}, below are some examples of cases when the condition (H2') holds.
\begin{enumerate}[i)]
\item
The scalar case ($m=1$).

\item
Two dimensional case ($n=2$).

\item
The coefficients of $L$ belong to VMO.
\end{enumerate}
\end{remark}

\begin{theorem}		\label{thm4.4a}
Assume the conditions (H1') and (H2').
Then, there exists the Green's function $\vec G(x,y)$ for \eqref{ERP}.
It is continuous in $\set{(x,y)\in\Omega\times\Omega: x\neq y}$.
Also, we have
\[
\vec G(y,x)= \vec G^{*}(x,y)^\top
\]
where $\vec G^{*}(x,y)$ is the Green's function of the adjoint problem \eqref{eq4.02e}.
If we further assume (H3'), then for any $x, y \in \Omega$ with $x\neq y$, we have the following pointwise estimates:
\begin{enumerate}[i)]
\item
$n=2$
\[
\abs{\vec G(x,y)} \le C\left\{1+\ln \left(\frac{\diam \Omega}{\abs{x-y}}\right)\right\},
\]

\item
$n\ge 3$
\[
\abs{\vec G(x,y)} \le C\abs{x-y}^{2-n},
\]
\end{enumerate}
where $C$ depends only on $n,m, \lambda, \diam \Omega$ and and parameters in (H1') and (H3').
\end{theorem}

The following lemma is parallel with Lemma~\ref{lem2.2pe} and the proof is similar.
\begin{lemma}		\label{lem4.2pe}
Let $\Omega$ be a Lipschitz domain.
Suppose 
\begin{equation}	\label{eq4.09ex}
\theta \in L^{p}(\partial \Omega)^{m\times m}, \;\text{ where  $p=n-1$ if  $n\ge 3$ and $p\in (1,\infty]$ if $n=2$}.
\end{equation}
Let $\Theta=M_\theta$, the operator of multiplication with $\theta$ defined by
\[
\ip{\Theta \vec u, \vec v} := \int_{\partial\Omega}  \theta \vec u \cdot \vec v \,dS,\quad \forall \vec u, \vec v \in H^{1/2}(\partial\Omega)^m.
\]
Suppose  that there is $\delta>0$ such that
\begin{equation}		\label{eq4.10ex}
\inf_{\vec e \in \bR^m,\; \abs{\vec e}=1} \left(\int_{\partial\Omega} \theta \,dS \right) \vec e \cdot \vec e\ge \delta> 0.
\end{equation}
If either $\theta \vec \xi \cdot \vec \xi \ge 0$ for all $\vec \xi\in \bR^m$ or $\norm{\theta}_{L^{p}}$ is sufficiently small, then (H1') holds.
\end{lemma}

The following corollaries are easy consequences of Theorem~\ref{thm4.4a} combined with Lemma~\ref{lem4.2pe}, Remark~\ref{rmk4.4pe}, and Theorem~\ref{thm3.3flu}.

\begin{corollary}
Let $m=1$ and $\Omega$ be a Lipschitz domain.
Suppose $\Theta=M_\theta$, where $\theta \ge 0$ and satisfies \eqref{eq4.09ex} and \eqref{eq4.10ex}.
Then there exists a (scalar) Green's function for \eqref{ERP} and it satisfies the conclusions of Theorems \ref{thm4.4a}.
\end{corollary}

\begin{corollary}
Let $\Theta=M_\theta$, where $\theta \in L^\infty(\partial\Omega)^{m\times m}$ is nonnegative definite and satisfies \eqref{eq4.10ex}.
Assume that one of the following holds:
\begin{enumerate}[(i)]
\item
$\Omega$ is a Lipschitz domain and in $\bR^2$.
\item
$\Omega$ is a $C^1$ domain in $\bR^n$ with $n\ge 3$ and the coefficients of $L$ belong to $VMO$.
\end{enumerate}
Then there exists the Green's function for \eqref{ERP} and it satisfies the conclusions of Theorem~\ref{thm4.4a}.
\end{corollary}

\section{Proofs of main theorems}		\label{D}

\subsection{Proof of Theorem \ref{B3.thm1}}
Let $Y=(y,s) \in Q$ and $\epsilon>0$ be fixed but arbitrary.
Fix any $a<s-\epsilon^2$ and let $\vec v_\epsilon=\vec v_{\epsilon; Y,k}$ be the weak solution in $\sV^{1,0}_2(\Omega\times(a,b))^m$ of the problem (see Lemma~\ref{B3.lem1})
\[
\left\{
\begin{aligned}
\sL\vec u=\tfrac{1}{\abs{Q^-_\epsilon}}1_{Q^-_\epsilon(Y)} \vec e_k &\quad \text{in }\, \Omega\times (a,b),\\
\partial \vec u/\partial \nu + \Theta \vec u=  0 &\quad \text{on }\, \partial \Omega\times (a,b),\\
\vec u=0 &\quad \text{on }\, \Omega\times \set{a},
\end{aligned}
\right.
\]
where $\vec e_k$ is the $k$-th unit vector in $\bR^m$.
By setting $\vec v_\epsilon(x,t)=0$ for $t<a$ and letting $b\to \infty$, we assume that $\vec v_\epsilon$ is defined on the entire $Q$.
Then, by \eqref{enieq}, we have
\begin{equation}		\label{mn.eq1c}
\tri{\vec v_\epsilon}_{\Theta; Q} \le C \epsilon^{-n/2}.
\end{equation}
Note that by \eqref{eq2.07ej}, $\vec v_\epsilon$ satisfies for all $t_1$ the identity
\begin{multline}			\label{51.eq1b1}
\int_\Omega \vec v_\epsilon (\cdot, t_1)\cdot \vec \phi(\cdot,t_1)\,dx-\int_{-\infty}^{t_1}\int_\Omega \vec v_\epsilon\cdot \vec \phi_t\,dX+\int_{-\infty}^{t_1}\int_\Omega A^{\alpha\beta}D_\beta \vec v_\epsilon\cdot D_\alpha \vec \phi\,dX\\
+\int_{-\infty}^{t_1} \ip{\Theta\vec v_\epsilon,\vec \phi}\,dt =\frac{1}{\abs{Q^-_\epsilon}} \int_{-\infty}^{t_1}\int_\Omega 1_{Q^-_\epsilon(Y)}  \phi^k\,dX,\quad \forall \vec \phi \in \sC^\infty(\bar Q)^m. 
\end{multline}
We define \emph{the averaged Green's function} $\vec \cG^\epsilon(\cdot,Y)=(\cG^\epsilon_{jk}(\cdot,Y))_{j,k=1}^m$ by setting
\[
\cG^\epsilon_{jk}(\cdot,Y)=v_\epsilon^j=v^j_{\epsilon;Y,k}.
\]
Next, for $\vec f \in \sC^\infty_c(\Omega\times (a,b))^m$, where $a<s<b$, let $\vec u$ be the weak solution in $\sV^{1,0}_2(\Omega\times (a,b))^m$ of the \emph{backward problem}
\begin{equation}		\label{mn.eq2}		
\left\{
\begin{aligned}
\sL^{*}\vec u=\vec f\quad &\text{in }\; \Omega\times (a,b)\\
\partial \vec u/\partial \nu^{*} + \Theta^{*} \vec u=  0 \quad&\text{on }\; \partial\Omega\times (a,b)\\
\vec u=  0 \quad&\text{on }\; \Omega\times \set{b}.
\end{aligned}
\right.
\end{equation}
By setting $\vec u(x,t)=0$ for $t>b$ and letting $a\to -\infty$, we assume that $\vec u$ is defined on the entire $Q$.
Then, we have
\begin{equation}		\label{mn.eq3}
\frac{1}{\abs{Q^-_\epsilon}} \int_{Q\cap Q_{\epsilon}^-(Y)} u^k(X)\,dX =\int_{Q} \cG^\epsilon_{ik}(X,Y) f^i(X)\,dX.
\end{equation}
Suppose $\vec f$ is supported in $Q^{+}_R(X_0)$, where $0<R<d_Y:=\dist(Y,\partial_p Q)$.
Then by \eqref{enieq}, we get
\begin{equation}		\label{mn.eq4b}
\tri{\vec u}_{\Theta; Q}\le C\norm{\vec f}_{\sL_{2(n+2)/(n+4)}(Q^{+}_R(X_0))}.
\end{equation}
By utilizing \eqref{mn.eq4b}, \eqref{eq2.19ar}, and (H2), and proceeding as in \cite[Section~3.2]{CDK}, we find that $\vec u$ is continuous in $Q_{R/4}^{+}(X_0)$ and satisfies (see \cite[Eq.(3.15)]{CDK})
\begin{equation}		\label{mn.eq4d}
\norm{\vec u}_{\sL_\infty(Q^{+}_{R/4}(X_0))} \le CR^{2}\norm{\vec f}_{\sL_\infty(Q^{+}_R(X_0))}.
\end{equation}
If $Q_\epsilon^-(Y) \subset Q^{+}_{R/4}(X_0)$, then by \eqref{mn.eq3} and \eqref{mn.eq4d}, we obtain
\[
\Abs{\int_{Q^{+}_R(X_0)} \cG^\epsilon_{ik}(\cdot,Y)f^i}\le \norm{\vec u}_{\sL_\infty(Q^{+}_{R/4}(X_0))} \le CR^2\norm{\vec f}_{\sL_\infty(Q^{+}_R(X_0))}.
\]
Therefore, by duality, we get
\[
\Norm{\vec \cG^\epsilon(\cdot,Y)}_{\sL_1(Q^{+}_R(X_0))}\le CR^2,
\]
provided $0<R<d_Y$ and $Q_\epsilon^-(Y) \subset Q_{R/4}^{+}(X_0)$.
Then by repeating the proof of \cite[Lemma~3.2]{CDK}, for any $X,Y\in Q$ satisfying $0<\abs{X-Y}_\sP<d_Y/6$, we have
\begin{equation}		\label{mn.eq4h}
\abs{\vec \cG^\epsilon(X,Y)}\le C\abs{X-Y}^{-n}_\sP, \quad  \forall \epsilon<\tfrac{1}{3}\abs{X-Y}_\sP.
\end{equation}

\begin{lemma}		\label{wse.lem1}
For any $Y\in Q$, $0<R<d_Y$, and $\epsilon>0$, we have
\begin{gather}		
\label{wse.eq2b}
\tri{\vec \cG^\epsilon(\cdot,Y)}_{Q\setminus Q_r(Y)} \le Cr^{-n/2},\\
\label{wse.eq2c}
\norm{\vec \cG^\epsilon(\cdot,Y)}_{\sL_{2(n+2)/n}(Q\setminus Q_r(Y))} \le Cr^{-n/2}.
\end{gather}
Also, for any $Y\in Q$ and $\epsilon>0$, we have
\begin{align}
\label{wse.eq2e}
\bigabs{\bigset{X\in Q : \abs{\vec \cG^\epsilon(X,Y)}>\tau}} &\le C \tau^{-\frac{n+2}{n}}, \quad \forall \tau> d_Y^{-n},\\
\label{wse.eq2d}
\bigabs{\bigset{X\in Q : \abs{D_x \vec \cG^\epsilon(X,Y)}>\tau}}&\le C \tau^{-\frac{n+2}{n+1}},  \quad \forall \tau> d_Y^{-(n+1)}.
\end{align}
Furthermore, for any $Y\in Q$, $0<R<d_Y$, and $\epsilon>0$, we have
\begin{align}
\label{wse.eq2h}
\norm{\vec \cG^\epsilon(\cdot,Y)}_{\sL_p(Q_R(Y))}&\le C_{p} R^{-n+(n+2)/p}\quad \text{for }\; p\in [1,\tfrac{n+2}{n}),\\
\label{wse.eq2i}
\norm{D \vec \cG^\epsilon(\cdot,Y)}_{\sL_p(Q_R(Y))} &\le C_{p} R^{-n-1+(n+2)/p} \quad \text{for }\; p\in[1,\tfrac{n+2}{n+1}).
\end{align}
\end{lemma}
\begin{proof}
For $0<r \le d_Y$, let $\eta$ be a smooth cut-off function satisfying
\begin{equation} 	\label{51.eq1h}
0\le \eta\le1, \quad \eta\equiv 1 \text{ on }Q_{r/2}(Y), \quad \eta\equiv 0 \text{ on } \bar Q_{r}(Y)^c, \quad \abs{D \eta}^2+\abs{\eta_t}\le 16r^{-2}.
\end{equation}
Suppose $\epsilon< r/6$ and denote $Q_t=\Omega\times (-\infty,t)$.
We set $\vec \phi=(1-\eta)^2\vec v_\epsilon$ in \eqref{51.eq1b1} and carry out a formal calculation to get for all $-\infty<t<\infty$, the identity
\begin{multline*}
\frac{1}{2} \int_\Omega \abs{(1-\eta) \vec v_\epsilon}^2(\cdot,t)+ \int_{Q_t}A^{\alpha \beta} D_\beta((1-\eta) \vec v_\epsilon) \cdot D_\alpha((1-\eta)\vec v_\epsilon)+\int_{-\infty}^{t_1} \ip{\Theta \vec v_\epsilon, (1-\eta)^2 \vec v_\epsilon}\\
 + \int_{Q_t} \left\{(1-\eta)\eta_t \abs{\vec v_\epsilon}^2 - D_\beta\eta D_\alpha \eta A^{\alpha\beta}\vec v_\epsilon\cdot \vec v_\epsilon+ (1-\eta)D_\beta \eta (A^{\alpha\beta}-A^{\beta\alpha}{}^*) \vec v_\epsilon \cdot D_\alpha \vec v_\epsilon\right\} =0.
\end{multline*}
The above computation is justified by means of a standard approximation technique involving Steklov average; see proof of Lemma~\ref{B3.lem1}.
By using Cauchy's inequality, we then obtain
\begin{multline}
\label{app.eq1c}
\sup_{-\infty< t<\infty} \frac{1}{2} \int_{\Omega} \abs{(1-\eta)\vec v_\epsilon}^2(\cdot, t)\,dx + \tilde{\lambda} \int_Q \abs{D((1-\eta)\vec v_\epsilon)}^2\,dX+\int_{-\infty}^\infty\ip{\Theta \vec v_\epsilon, (1-\eta)^2 \vec v_\epsilon}\,dt \\
\le C(\lambda,\tilde{\lambda}) r^{-2} \int_{Q_r(Y)\setminus Q_{r/2}(Y)} \abs{\vec v_\epsilon}^2 \, dX.
\end{multline}
Note that $(1-\eta)^2 \vec v_\epsilon(\cdot,t) \equiv  \vec v_\epsilon(\cdot,t) \equiv (1-\eta) \vec v_\epsilon(\cdot,t)$ in $H^{1/2}(\partial\Omega)^m$ for a.e. $t$.
Then, we derive from \eqref{app.eq1c} and \eqref{mn.eq4h} that
\begin{equation}			\label{eq5.13hu}
\tri{(1-\eta)\vec v_\epsilon}_{\Theta; Q}^2 \le C r^{-n}.
\end{equation}
On the other hand, by \eqref{vtq}, \eqref{eq2.19ar}, and \eqref{mn.eq1c}, we have
\[
\tri{(1-\eta)\vec v_\epsilon}_{\Theta; Q}^2 \le \tri{\vec v_\epsilon}_{\Theta; Q}^2+ \tilde{\lambda} \int_{Q} \abs{D\vec v_\epsilon}^2 \,dX+ 2\tilde{\lambda} \int_{Q_r(Y)} \abs{D \eta}^2 \abs{\vec v_\epsilon}^2 \,dX \le C \tri{\vec v_\epsilon}_{\Theta; Q}^2  \le C \epsilon^{-n}.
\]
Therefore, the estimate \eqref{eq5.13hu} holds for all $\epsilon >0$.
Then, \eqref{wse.eq2b} and \eqref{wse.eq2c} follow from \eqref{eq5.13hu}, \eqref{eq2.19ar}, \eqref{eq2.11ar}, and the fact that $d_Y/6$ and $d_Y$ are comparable to each other.
We derive \eqref{wse.eq2h} and \eqref{wse.eq2i}, respectively, from \eqref{wse.eq2e} and  \eqref{wse.eq2d}, which in turn follow respectively from  \eqref{wse.eq2c} and \eqref{wse.eq2b}; see \cite[Lemmas~3.3 and 3.4]{CDK}.
\end{proof}

\begin{lemma}
Let $\set{u_k}_{k=1}^\infty$ be a sequence in $\sV_2(Q)$.
If $\sup_k \tri{u_k}_{Q} \le A$, then there exist a subsequence $\set{u_{k_j}}_{j=1}^\infty \subseteq \set{u_k}_{k=1}^\infty$ and $u\in \sV_2(Q)$ with $\tri{u}_{Q}\le A$ such that $u_{k_j} \rightharpoonup u$ weakly in  $\sW^{1,0}_2(\Omega\times (a,b))$ for all $-\infty<a<b<\infty$.
\end{lemma}
\begin{proof}
See \cite[Lemma~A.1]{CDK}.
\end{proof}

The above two lemmas contain all the ingredients for the construction of a function $\vec \cG(\cdot,Y)$ such that for a sequence $\epsilon_\mu$ tending to zero, we have
\begin{align}
\label{eq5.27a}
\vec \cG^{\epsilon_\mu}(\cdot,Y) &\rightharpoonup \vec \cG(\cdot,Y) \;\text{ weakly in }\, \sW^{1,0}_q(Q_{d_Y}(Y))^{m \times m},\\
\label{eq5.28b}
(1-\eta)\vec \cG^{\epsilon_\mu}(\cdot,Y) &\rightharpoonup (1-\eta)\vec \cG(\cdot,Y) \; \text{ weakly in }\,\sW^{1,0}_2(\Omega\times (-T,T))^{m\times m}, \;\; \forall T>0,
\end{align}
where $1<q<\frac{n+2}{n+1}$ and $\eta$ is as defined \eqref{51.eq1h} with $r= d_Y/2$.
It is routine to verify that $\vec \cG(\cdot,Y)$ satisfies the same estimates as in Lemma~\ref{wse.lem1}; see \cite[Section~4.2]{CDK}.
Then, it is clear that $\vec \cG(\cdot, Y)$ satisfies the property a) in Section~\ref{2.rp}.
We shall now show that $\vec \cG(X,Y)$ also satisfies the properties b) and c) so that $\vec \cG(X,Y)$ is indeed the Green's function for \eqref{RP}.
To verify the property b), let us assume $\vec \phi \in \sC^\infty_c(\bar \Omega\times (-T,T))^m$, where $-T<s<T$.
Then for $\eta$ satisfying \eqref{51.eq1h} with $r=d_Y$, we get from \eqref{51.eq1b1} that 
\begin{multline}		\label{51.eq2}
\frac{1}{\Abs{Q^-_{\epsilon_\mu}}} \int_{Q \cap Q^-_{\epsilon_\mu}(Y)}  \phi^k=
\int_Q A^{\alpha\beta} D_\beta \vec \cG_{\cdot k}^{\epsilon_\mu}(\cdot,Y) \cdot D_\alpha ((1-\eta)\vec \phi)+\int_Q A^{\alpha\beta}D_\beta\vec \cG_{\cdot k}^{\epsilon_\mu}(\cdot,Y) \cdot D_\alpha (\eta\vec \phi)\\
-\int_Q \vec \cG_{\cdot k}^{\epsilon_\mu}(\cdot,Y) \cdot (\eta\vec \phi)_t - \int_Q \vec \cG_{\cdot k}^{\epsilon_\mu}(\cdot,Y)\cdot ((1-\eta)\vec \phi)_t
+\int_{-T}^T \ip{\vec \Theta \vec \cG_{\cdot k}^{\epsilon_\rho}(\cdot,Y) ,\vec \phi}.
\end{multline}
Observe that $\vec u \mapsto \int_{-T}^{T}\ip{\Theta\vec u,\vec \phi}$ is a bounded linear functional on $\sW^{1,0}_2(\Omega\times (-T,T))^m$.
Therefore, by using \eqref{eq5.27a} and \eqref{eq5.28b}, and taking $\mu\to \infty$ in \eqref{51.eq2}, we verify the property b); see \cite[p. 1662]{CDK} for the details.
To verify the property c), let us assume that $\vec f$ is supported in $\bar \Omega\times (a,b)$, where $a<s<b$ and $\tilde{\vec u}$ be the unique weak solution in $\sV^{1,0}_2(\Omega\times (a,b))^m$ of the problem \eqref{mn.eq2}.
By setting $\tilde{\vec u}(x,t)=0$ for $t>b$ and letting $a\to -\infty$, we may assume that $\tilde{\vec u}$ is defined on the entire $Q$.
Then, similar to \eqref{mn.eq3}, we have
\[
\frac{1}{\abs{Q^-_\epsilon}} \int_{Q\cap Q_{\epsilon}^-(Y)} \tilde{u}^k(X)\,dX =\iint_{Q} \cG^{\epsilon_\mu}_{ik}(X,Y) f^i(X)\,dX.
\]
By the condition (H2), it follows that $\tilde{\vec u}$ is locally H\"older continuous in $Q$; see the remark we made in deriving \eqref{mn.eq4d}.
By writing $\vec f=\zeta \vec f+ (1-\zeta) \vec f$ and using \eqref{eq5.27a}, \eqref{eq5.28b}, and taking the limit $\mu\to \infty$, we then get 
\[
\tilde{\vec u}(Y)=\int_Q \vec \cG(X,Y)^\top \vec f(X)\,dX.
\]
Therefore, we have $\tilde{\vec u}\equiv \vec u$ and thus the property c) is verified.

It is clear from the construction that $\vec \cG(x,t,y,s)\equiv 0$ if $t<s$.
By a similar argument as above, we obtain the Green's function $\vec \cG^{*}(\cdot, X)$ for the adjoint problem \eqref{RPs} that satisfies the natural counterparts of  the properties of the Green's function for \eqref{RP}.
Note that the condition (H2) together with the estimates i), ii) listed in Remark~\ref{rmk3.8} and its counterparts imply that $\vec \cG(\cdot,Y)$ and $\vec \cG^{*}(\cdot,X)$ are locally H\"older continuous in $Q\setminus \set{Y}$ and $Q\setminus\set{X}$, respectively.
Using the continuity discussed above and proceeding similar to \cite[Lemma~3.5]{CDK}, we find that
\begin{equation}		\label{ctn.eq1}
\vec \cG(Y,X)=\vec \cG^{*}(X,Y)^\top, \quad \forall X,Y\in Q, \quad X\neq Y.
\end{equation}
Moreover, similar to \cite[Eqs.~(3.44) and (3.45)]{CDK}, we have
\[
\vec \cG^\epsilon(X,Y)=\frac{1}{\abs{Q^-_\epsilon}} \int_{Q\cap Q_{\epsilon}^-(Y)}  \vec \cG(X,Z)\,dZ
\]
which justifies why we call it \emph{the averaged Green's function}, and
\begin{equation}		\label{eq5.30x}
\lim_{\epsilon\to 0+} \vec \cG^\epsilon(X,Y)=\vec \cG(X,Y).
\end{equation}
By \eqref{ctn.eq1} and the counterpart of the property c) in Section~\ref{2.rp}, we see that $\vec u$ defined by the formula \eqref{eq3.04c} is a weak solution in $\sV^{1,0}_2(Q)^m$ of \eqref{eq3.05a}.

We now prove the identity \eqref{EP.1b} for the weak solution in $\sV^{1,0}_2(\Omega\times(s,\infty))^m$ of the problem \eqref{EP.1b1}.
Let $X=(x,t) \in Q$ with $t>s$ be fixed. 
Then, similar to \eqref{mn.eq3}, for $\epsilon$ sufficiently small, we have (see \cite[Eq.~(3.49)]{CDK})
\[
\frac{1}{\abs{Q^{+}_\epsilon}} \int_{ Q_{\epsilon}^{+}(X)}   u^k(Y)\,dY =\int_{\Omega} \hat{\cG}^\epsilon_{ik}(y,s,x,t) \psi^i(y)\,dy,
\]
where $\hat{\vec \cG}{}^\epsilon(\cdot,X)$ is the averaged Green's function for \eqref{RPs}.
The condition (H2) implies that $\vec u$ is continuous in $\Omega\times (s,\infty)$, and thus we have
\[
\lim_{\epsilon \to 0} \frac{1}{\abs{Q^{+}_\epsilon}} \int_{Q_{\epsilon}^{+}(X)}   u^k(Y)\,dY = u^k(X).
\]
On the other hand,  by \eqref{eq5.30x} and the counterparts of \eqref{wse.eq2b}, together with the dominated convergence theorem, we get
\[
\lim_{\epsilon \to 0} \int_{\Omega} \hat{\cG}^\epsilon_{ik}(y,s,x,t) \psi^i(y)\,dy=\int_{\Omega} \cG^{*}_{ik}(y,s,x,t) \psi^i(y)\,dy.
\]
Then, the identity \eqref{EP.1b} follows from \eqref{ctn.eq1}.
Finally, we obtain \eqref{eq3.07c} similar to \eqref{mn.eq4h} and get \eqref{eq3.08d} from \eqref{eq3.07c} and the condition (H2).
\hfill\qedsymbol

\subsection{Proof of Theorem \ref{B3.thm2}}
 Let $\psi$ be a bounded Lipschitz function on $\bR^n$ satisfying $
\abs{D\psi}\le M$ a.e. for some $M>0$ be chosen later.
For $s<t$, we define an operator $P^\psi_{s\to t}$ on $L^2(\Omega)^m$ as follows:
For $\vec f \in L^2(\Omega)^m$, fix any $T$ such that $T>t$ and let $\vec u$ be the unique weak solution in $\sV^{1,0}_2(\Omega\times(s,T))^m$ of the problem
\[
\left\{
\begin{array}{ll}
\sL\vec u= 0 &\text{in }\; \Omega\times (s,T)\\
\partial \vec u/\partial \nu+\Theta \vec u=0  &\text{on }\;\partial\Omega \times (s,T)\\
\vec u=e^{-\psi}\vec f &\text{on }\; \Omega\times \set{s}
\end{array}
\right.
\]
and define $P^\psi_{s\to t} \vec f(x) := e^{\psi(x)}\vec u(x,t)$.
Then, by \eqref{EP.1b}, we find
\[
P^\psi_{s\to t}\vec f(x)=
e^{\psi(x)}\int_{\Omega}\vec \cG(x,t,y,s)e^{-\psi(y)}\vec f(y)\,dy.
\]
By usual approximation involving Steklov average (see the proof of Lemma~\ref{B3.lem1}), for a.e. $0<t<s$, we get
\begin{multline*}
\frac{d}{dt} \int_\Omega  \frac{1}{2} e^{2\psi}\abs{\vec u(\cdot,t)}^2\,dx + \int_\Omega e^{2\psi}A^{\alpha\beta}D_\beta \vec u(\cdot,t) \cdot D_\alpha \vec u(\cdot,t)\,dx +\ip{\Theta \vec u(\cdot, t),e^{2\psi} \vec u(\cdot,t)}\\
=\int_\Omega  2 e^{2\psi} D_\alpha \psi A^{\alpha\beta}D_\beta \vec u(\cdot,t) \cdot  \vec u(\cdot,t)\,dx.
\end{multline*}
Therefore, by using \eqref{eq2.09local} and Cauchy's inequality, we get
\[
\frac{d}{dt} \int_\Omega  \frac{1}{2} e^{2\psi}\abs{\vec u(\cdot,t)}^2\,dx + \tilde{\lambda} \int_\Omega e^{2\psi}\abs{D \vec u(\cdot,t)}^2 \,dx
\le \nu M^2 \int_\Omega  e^{2\psi} \abs{\vec u(\cdot,t)}^2 \,dx,
\]
where $\nu=\nu(\lambda, \tilde{\lambda})$.
Therefore,  $I(t):=\int_\Omega e^{2\psi(x)}\abs{\vec u(x,t)}^2\,dx$ satisfies
\[
I'(t) \le  \nu M^2 I(t)\quad\text{for a.e. $0<t<s$},
\]
and thus, we obtain
\begin{equation} \label{eq3.70}
\norm{P^\psi_{s\to t}\vec f}_{L^2(\Omega)} \le e^{\nu M^2 (t-s)}\norm{\vec f}_{L^2(\Omega)}.
\end{equation}
As we pointed out in Remark~\ref{rmk2.5pe}, we may assume that $R_1=\diam\Omega$.
We set $R= \min(\sqrt{t-s}, \diam \Omega)$ and use the condition (H3) to estimate
\begin{align*}
e^{-2\psi(x)}\abs{P^\psi_{s\to t}\vec f(x)}^2  &= \abs{\vec u(x,t)}^2\\
& \le A_1^2 R^{-(n+2)} \int_{t-R^2}^t \int_{\Omega}e^{-2\psi(y)} \abs{P^\psi_{s\to\tau}\vec f(y)}^2\,dy\,d\tau,\quad \forall x \in \Omega.
\end{align*}
Thus, by using \eqref{eq3.70}, we derive
\begin{align*}
\abs{P^\psi_{s\to t}\vec f(x)}^2 &\le A_1^2 R^{-n-2} \int_{t-R^2}^t \int_{\Omega}e^{2MR} \abs{P^\psi_{s\to\tau}\vec f(y)}^2\,dy\,d\tau\\
& \le A_1^2 R^{-n-2} \, e^{2MR} \int_{t-R^2}^t e^{2 \nu M^2 (\tau-s)}\norm{\vec f}_{L^2(\Omega)}^2\,d\tau\\
& \le A_1^2 R^{-n}\,e^{2MR+2 \nu M^2(t-s)}\norm{\vec f}_{L^2(\Omega)}^2,\quad \forall x \in \Omega.
\end{align*}
We have thus obtained the following $L^2\to L^\infty$ estimate for $P^\psi_{s\to t}$:
\[
\norm{P^\psi_{s\to t}\vec f}_{L^\infty( \Omega)} \le  A_1 R^{-n/2}\,e^{MR+\nu M^2 (t-s)} \norm{\vec f}_{L^2(\Omega)},
\]
which corresponds to \cite[Eq.~(5.33)]{CK2}.
The rest of proof is identical to that of \cite[Theorem~3.21]{CK2} and omitted.
\hfill\qedsymbol

\subsection{Proof of Theorem~\ref{thm3.3flu}}

\subsubsection{Proof of Case (i)}
We follow De Giorgi's method.
We remark that an elliptic version of estimate \eqref{30eq2a} is proved in \cite[Lemma~3.1]{LS04}.
For $i=1,2,\ldots$, let
\[
R_i=\frac{R}{2}+\frac{R}{2^i}, \quad k_i=k\left(2-\frac{1}{2^{i-1}}\right) \quad \text{and }\; A_i=\Set{X\in  Q^{-}_{R_i}(X_0)\cap Q: u(X)>k_i},
\]
where $k \ge 0$ to be chosen later, and let $\eta$ be a smooth function in $\bR^{n+1}$ satisfying 
\[
0\le \eta \le1, \;\; \supp \eta \subset Q_{R_i}(X_0), \;\; \eta \equiv 1 \text{ on }Q_{R_{i+1}}(X_0)\;\; \text{and}\;\; \abs{\partial_t\eta}+\abs{D_x \eta}^2\le \frac{4^{i+3}}{R^2}.
\]
Testing with $\eta^2(u-k_{i+1})_+$ in \eqref{eqj.1b}, we get for a.e. $t \in (-T,0)$ that
\begin{multline*}
\frac{d}{dt}\int_\Omega  \frac{1}{2}\eta^2 (u-k_{i+1})_+^2+\int_Q \eta^2 A^{\alpha\beta} D_\beta(u-k_{i+1})_+ D_\alpha(u-k_{i+1})_{+} +\int_{S} \theta \eta^2 u(u-k_{i+1})_+\\
+\int_Q 2\eta (u-k_{i+1})_{+} A^{\alpha\beta} D_\beta(u-k_{i+1})_+ D_\alpha \eta  - \int_Q \eta \partial_t \eta (u-k_{i+1})_+^2=0.
\end{multline*}
Then by using $\theta \ge 0$ and Cauchy's inequality, we get
\begin{equation}	
\label{B1.eq1a}
\tri{(u-k_{i+1})_+}_{Q^{-}_{R_{i+1}}(X_0) \cap Q}^2 \le C\frac{4^i}{R^2}\int_{Q^-_{R_i}(X_0)\cap Q}(u-k_{i+1})_+^2
\le C\frac{4^i}{R^2}\int_{A_i}(u-k_{i})_+^2.
\end{equation}
Denote
\[
Y_i=\int_{A_i}(u-k_i)^2_+
\]
and observe that 
\[
Y_i \ge \int_{A_{i+1}}(u-k_i)_+^2 \ge \int_{A_{i+1}}(k_{i+1}-k_i)^2=\frac{k^2}{4^i}\abs{A_{i+1}},
\]
Then by \eqref{2.eq4d} and \eqref{B1.eq1a}, we obtain
\begin{align*}
Y_{i+1}&\le \norm{(u-k_{i+1})_+}_{\sL_{2(n+2)/n}(A_{i+1})}^2\,\abs{A_{i+1}}^{2/(n+2)} \le C\frac{4^i}{R^2}Y_i \left(\frac{4^i}{k^2}Y_i\right)^{2/(n+2)}\\
&\le C \frac{16^i}{k^{4/(n+2)} R^2} Y_i^{1+2/(n+2)} =: 16^i K Y_i^{1+\sigma},
\end{align*}
where we have set $K=C/k^{4/(n+2)}R^2$ and $\sigma=2/(n+2)$.
Now, we choose
\[
k=C^{(n+2)/4} 2^{(n+2)^2/2}R^{-(n+2)/2}\norm{u}_{\sL_2(Q_R^-(X_0)\cap Q)}
\]
so that we have
\[
Y_1\le \int_{Q_R^-(X_0)\cap Q} \abs{u}^2\,dX =16^{-1/\sigma^2}K^{-1/\sigma}.
\]
Then by \cite[Lemma~15.1, p. 319]{DiB}), we have $Y_i\to 0$ as $i\to \infty$, and thus, we get
\[
u\le 2k \quad \text{on }\;  Q^-_{R/2}(X_0)\cap Q.
\]
By applying the same argument to $-u$, we obtain \eqref{30eq2a}.
\hfill\qedsymbol

\subsubsection{Proof of Case (ii)}
Let $0<R<\min(\sqrt{b-a}, r_0)$ be arbitrary but fixed, where $r_0 \le 16$ is to be determined.
For any $Y\in Q^{-}_{R/2}(X_0)\cap Q$ and $0<\rho<r\le R/16$, we choose a function $\zeta$ such that 
\[
0\le \zeta\le 1, \quad \supp \zeta\subset Q_{(\rho+r)/2}(Y), \quad \zeta\equiv 1 \text{ on } Q_\rho(Y), \quad \abs{\partial_t\zeta}+\abs{D_x\zeta}^2\le 32(r-\rho)^{-2}.
\]
Then $\vec v=\zeta \vec u$ becomes a weak solution of the problem
\[
\left\{
\begin{aligned}
\sL \vec v=\vec \Psi-D_\alpha \vec F_\alpha &\quad \text{in }\, Q,\\
\partial \vec v/\partial \nu=\vec g+n_\alpha \vec F_\alpha &\quad \text{on }\, S,\\
\vec v(\cdot,a)=0 &\quad \text{on }\, \Omega,
\end{aligned} 
\right.
\]
where we set
\begin{equation}	\label{eq5.24sn}
\vec \Psi=\partial_t \zeta \vec u-D_\alpha \zeta  A^{\alpha\beta}D_\beta \vec u, \quad \vec F_\alpha =D_\beta \zeta A^{\alpha \beta}\vec u, \quad \text{and}\quad \vec g=-\zeta \theta \vec u.
\end{equation}
Let us denote
\begin{equation}		\label{eq6.13yy}
\vec c(t)=\int_\Omega \vec\Psi(x,t)\,dx+ \int_{\partial\Omega}  \vec g(x,t)\,dS_x
\end{equation}
and let $\vec V(\cdot, t)$ be a unique (up to a constant) weak solution of the Neumann problem
\[
\left\{
\begin{aligned}
\Delta \vec V(\cdot,t)&=\vec \Psi(\cdot,t) - \abs{\Omega}^{-1} \vec c(t)\quad \text{in }\; \Omega,\\
\partial \vec V(\cdot,t)/\partial n&=- \vec g(\cdot,t) \quad \text{on }\; \partial\Omega.
\end{aligned}
\right.
\]
We assume that $\vec V$ is constructed in such a way that it is measurable in $t$.
Then, by \cite[Corollary~9.3]{FMM} together with the embedding theorems of Sobolev and Besov spaces (see e.g., \cite{BL}), we have the following estimate for $D\vec V(\cdot,t)$
\[
\norm{D \vec V(\cdot,t)}_{L^p(\Omega)} \le C \left(\norm{\vec \Psi(\cdot,t)}_{L^{pn/(p+n)}(\Omega)}+\norm{\vec g(\cdot,t)}_{L^{p(n-1)/n}(\partial\Omega)}\right),
\]
where $C= C(n,m,p,\Omega)$, and thus, we get
\begin{equation}		\label{19eq1c}
\norm{D\vec V}_{\sL_{p,q}(Q)}\le C\left(\norm{\vec \Psi}_{\sL_{np/(n+p),q}(Q)}+\norm{\vec g}_{\sL_{p(n-1)/n,q}(S)}\right).
\end{equation}
Notice that if we set $\vec h_\alpha=D_\alpha\vec V-\vec F_\alpha$, then $\vec v$ becomes a weak solution of the problem
\[
\left\{
\begin{aligned}
\sL \vec v = D_\alpha \vec h_\alpha +\abs{\Omega}^{-1}\vec c(t) &\quad \text{in }\;Q,\\
\partial \vec v/\partial \nu + n_\alpha \vec h_\alpha =  0&\quad \text{on }\;S,\\
\vec v(\cdot, a)=0&\quad \text{on }\; \Omega.
\end{aligned}
\right.
\]
Note that with $\vec \Psi$ and $\vec g$ as given in \eqref{eq5.24sn}, the function $\vec c(t)$  in \eqref{eq6.13yy} satisfies $\vec c(a)=0$ and the identity
\[
-\int_{a}^b \left(\int_\Omega \vec v(x,t)\,dx\right)\cdot \vec \phi'(t)\,dt=\int_{a}^b \vec c(t)\cdot \vec \phi(t)\,dt,\quad \forall \vec \phi(t)\in C^\infty_c(-T,0)^m
\]
so that $\vec c(t)$ has the weak derivative $\vec c'(t)=\int_\Omega \vec v(x,t)\,dx$.
Therefore, we have
\begin{equation}		\label{19eq1b}
\sup_{a<t<b}\abs{\vec c(t)}\le \int_Q \abs{\vec v}\,dX.
\end{equation}
We then apply \cite[Theorem 8.1]{DKd11}, \eqref{19eq1c}, and \eqref{19eq1b} to conclude that 
\begin{equation}		\label{19eq1e}
\norm{D\vec v}_{\sL_{p,q}(Q)}
\le C\left(\norm{\vec \Psi}_{\sL_{np/(n+p),q}(Q)}+\norm{\vec F_{\cdot}}_{\sL_{p,q}(Q)}+\norm{\vec g}_{\sL_{(n-1)p/n,q}(S)}+\norm{\vec v}_{\sL_1(Q)}\right),
\end{equation}
where $C$ depends only on $n, m, \lambda,Q, p, q$ and $A^{\alpha\beta}$.

Hereafter in the proof, we shall use the following notation
\begin{equation}		\label{lcd}
\begin{aligned}
U_r(X)=Q_r(X)\cap Q,&\quad  S_r(X)=Q_r(X)\cap S, \\ 
U^{-}_r(X)=Q^{-}_r(X)\cap Q,&\quad  S_r^{-}(X)=Q^{-}_r(X)\cap S, \\ \Omega_r(x)=B_r(x) \cap \Omega,&\quad \Sigma_r(x)=B_r(x)\cap \partial \Omega,
\end{aligned}
\end{equation}
and shall drop the reference to $X$ or $x$ if it is clear from the context.
We recall the following version of localized Sobolev inequality:
For $1\le p<n$, there exists $C'=C'(n,p,\Omega)$ such that for any $u\in W^{1,p}(\Omega_{r}(x))$ with $x\in\Omega$ and $0<r<\diam \Omega$, we have
\begin{multline}		\label{eq5.31cr}
\norm{u}_{L^{(n-1)p/(n-p)}(\Sigma_{s})} + \norm{u}_{L^{np/(n-p)}(\Omega_s)}\\
\le C' \left((r-s)^{-1} \norm{u}_{L^p(\Omega_{r})}+\norm{Du}_{L^p(\Omega_{r})}\right),\quad \forall s \in(0,r).
\end{multline}

By the properties of $\zeta$, H\"older's inequality, and \eqref{eq5.31cr} with $np/(n+p)$ and $(\rho+r)/2$ in place of $p$ and $s$, we get from \eqref{eq5.24sn} that
\begin{gather*}
\norm{\vec \Psi}_{\sL_{np/(n+p),q}(Q)} \le  Cr (r-\rho)^{-2} \norm{\vec u}_{\sL_{p,q}(U_r)} + C (r-\rho)^{-1} \norm{D\vec u}_{\sL_{np/(n+p),q}(U_r)},\\
\norm{\vec F_{\cdot}}_{\sL_{p,q}(Q)} \le C(r-\rho)^{-1} \norm{\vec u}_{\sL_{p,q}(U_r)},\\
\norm{\vec g}_{\sL_{(n-1)p/n,q}(S)} \le C \norm{\vec u}_{\sL_{(n-1)p/n,q}(S_{(\rho+r)/2})} \le C C' \left((r-\rho)^{-1}\norm{\vec u}_{\sL_{np/(n+p),q}(U_{r})}+ \norm{D\vec u}_{\sL_{np/(n+p),q}(U_{r})}\right),\\
\norm{\vec v}_{\sL_1(Q)} \le C r^{n+2-n/p-2/q} \norm{\vec u}_{\sL_{p,q}(U_r)}.
\end{gather*}
Therefore, we get from \eqref{19eq1e} that (recall that $r \le r_0/16 \le 1$)
\begin{equation}		\label{13-eq1a}
\norm{D\vec u}_{\sL_{p,q}(U_\rho)}\le  C''_p \frac{r}{(r-\rho)^2}\left(\norm{\vec u}_{\sL_{p,q}(U_{r})}+\norm{D\vec u}_{\sL_{np/(n+p),q}(U_{r})}\right)
\end{equation}
where $C''_p=C''_p(n, m, \lambda,Q, \norm{\theta}_\infty, p,q, A^{\alpha\beta})$.

The proof of the following lemma will be given in Section~\ref{app}.

\begin{lemma}		\label{lem2.3pe}
Let $\Omega$ be a Lipschitz domain and let $\vec u$ be a weak solution of
\[
\left\{
\begin{aligned}
\vec u_t-D_\alpha(A^{\alpha\beta}D_\beta \vec u)=0 &\quad \text{in }\, Q=\Omega\times (a,b),\\
\partial \vec u/\partial \nu+\theta\vec u=0 &\quad \text{on }\, S=\partial \Omega\times (a,b),
\end{aligned}
\right.
\]
where $\theta \in \sL_\infty(S)^{m \times m}$.
Then, there exist constants $q_0>2$ and $r_0>0$ such that for any $x_0\in \overline \Omega$ and $0<r<\min(\sqrt{b-a}, r_0)$, we have
\begin{multline}		\label{drowsy}
r^{-(n+2)/q_0}\left(\int_{U^{-}_{r/2}(X_0)}\abs{D\vec u}^{q_0}+\abs{\vec u}^{q_0}\,dX\right)^{1/q_0}\le C r^{-(n+2)/2}\left(\int_{U^{-}_r(X_0)}\abs{D\vec u}^2+\abs{\vec u}^2\,dX\right)^{1/2}\\
+Cr^{-(n+2)/2} \norm{\vec u}_{\sL_{2,\infty}(U^{-}_r(X_0))};\quad X_0=(x_0,b),
\end{multline}
where $C>0$ is a constant depending only on $n, m, \lambda, \Omega$ and $\norm{\theta}_{\infty}$.
\end{lemma}

Now, take $r_0>0$ and $q=q_0>2$ as given in Lemma~\ref{lem2.3pe}.
By replacing $r_0$ by $\min(r_0,8)$, we may assume that $r_0\le 8$.
We choose $p>q$ such that
\[
\mu:=1-2/q-n/p>0.
\]
We fix $k$ to be the smallest integer satisfying $k\ge n(1/2-1/p)$ and set 
\[
p_i=\frac{np}{n+pi} \quad \text{and}\quad r_i=\rho+\frac{(r-\rho)}{k} i, \qquad i=0,1,\cdots, k
\]
and 
\[
\tilde{C}= \max_{0\le i \le k} C''_{p_i},\quad \text{where $C''_p$ is as appears  in \eqref{13-eq1a}}.
\]
Then, we apply \eqref{13-eq1a} iteratively (set $\rho=r_{i}$, $r=r_{i+1}$, and $p=p_i$) and to get 
\[
\norm{D\vec u}_{\sL_{p,q}(U_{\rho})}\le \sum_{i=1}^k \tilde{C}^i r^i\left(\frac{k}{r-\rho}\right)^{2i}\norm{\vec u}_{\sL_{p_{i-1},q}(U_{r_i})}+\tilde{C}^k r^k \left(\frac{k}{r-\rho}\right)^{2k}\norm{D\vec u}_{\sL_{p_k,q}(U_{r})}.
\]
Notice that $1<p_k\le 2<q$.
By using H\"older's inequality, we then obtain
\begin{align}
\nonumber
\rho^{-n/2-\mu}\norm{D\vec u}_{\sL_2(U_\rho)}
&\le \abs{B_1}^{1/2-1/p}\norm{D\vec u}_{\sL_{p,q}(U_\rho)}\\
\label{5.56cacao}
& \le C\left(\frac{r}{r-\rho}\right)^{2k} \left( r^{-1}\norm{\vec u}_{\sL_{p,q}(U_{r})}
+ r^{1-\mu-(n+2)/q}\norm{D\vec u}_{\sL_{q}(U_{r})} \right).
\end{align}
Note that if we denote $Y_0=(y, \hat{s})$ with $\hat{s}=\min(s+r^2,b)$, then we have
\[
U_r=U_r(Y) \subset U_{2r}^{-}(Y_0),\quad U_{8r}^-(Y_0)\subset U_R^-(X_0).
\]
Therefore, if we take $\rho<r/2$ in \eqref{5.56cacao}, then by Lemma \ref{lem2.3pe} followed by Caccioppoli type inequality (cf. \eqref{app.eq1c} -- \eqref{eq5.13hu}), we have
\begin{equation}		\label{eq5.34hungry}
\int_{U_\rho}\abs{D\vec u}^2\le C \rho^{n+2\mu} \left( \frac{1}{r^2} \norm{\vec u}_{\sL_{p,q}(U_r)}^2+\frac{1}{r^{n+2+2\mu}} \int_{U_{8r}^{-}(Y_0)}\abs{\vec u}^2 \right),
\end{equation}
where we again used that $r \le r_0/16\le 1$.
\begin{lemma}		\label{13-lem1}
Under the same hypothesis of Lemma~\ref{lem2.3pe}, we have 
\[
\int_{U^{-}_{r}(X_0)}\abs{\vec u- \vec u_{X_0,r}}^2\,dX\le C r^2\left(\int_{U^{-}_r(X_0)}\abs{\vec u}^2+\abs{D\vec u}^2\,dX\right);\quad \vec u_{X_0,r}=\fint_{U^{-}_{r}(X_0)} \vec u,
\]
where $C>0$ is a constant depending only on $n, m, \lambda,\Omega$ and $\norm{\theta}_\infty$.
\end{lemma}
The proof of Lemma~\ref{13-lem1} will be given Section~\ref{app}.
By Lemma~\ref{13-lem1} combined with \eqref{eq5.34hungry} and H\"older's inequality, we obtain
\begin{align*}
\int_{U_{\rho}^{-}(Y)}\abs{\vec u- \vec u_{Y,\rho}}^2 
&\le C\rho^2 \int_{U_\rho^-(Y)}\abs{\vec u}^2+C \rho^{n+2+2\mu} \left( \frac{1}{r^2} \norm{\vec u}_{\sL_{p,q}(U_r(Y))}^2+\frac{1}{r^{n+2+2\mu}} \int_{U_{8r}^-(Y_0)} \abs{\vec u}^2 \right)\\
&\le C \rho^{n+2+2\mu} \left( (1+r^{-2})\norm{\vec u}_{\sL_{p,q}(U_{r}(Y))}^2+r^{-(n+2+2\mu)} \int_{U_{8r}^-(Y_0)}\abs{\vec u}^2\right)\\
&\le C \rho^{n+2+2\mu} \left( r^{-2}\norm{\vec u}_{\sL_{p,q}(U_{2r}^-(Y_0))}^2+r^{-(n+2+2\mu)} \int_{U_{8r}^-(Y_0)}\abs{\vec u}^2\right).
\end{align*}
Now, we take $r=R/16$ in the above.
Then, from the above inequality we conclude that for any $Y\in U_{R/2}^-(X_0)$ and $0<\rho<r/2=R/32$, we have
\[
\int_{U_{\rho}^-(Y)}\abs {\vec u-\vec u_{Y,\rho}}^2\le C \rho^{n+2+2\mu} \left(R^{-2} \norm{\vec u}_{\sL_{p,q}(U_{R}^-(X_0))}^2+R^{-(n+2+2\mu)} \norm{\vec u}_{\sL_2(U_R^-(X_0))}^2 \right).
\]
On the other hand, it is easy to see that for $R/32 \le \rho \le  R/2$, we have
\[
\int_{U^{-}_{\rho}(Y)}\abs{\vec u- \vec u_{Y,\rho}}^2 \le C\int_{U^{-}_\rho(Y)}\abs{\vec u}^2
\le C\int_{U^{-}_R(X_0)}\abs{\vec u}^2 \le  C\left(\frac{\rho}{R}\right)^{n+2+2\mu}\int_{U^{-}_R(X_0)}\abs{\vec u}^2,
\]
and thus, for any $Y\in U^{-}_{R/2}(X_0)$ and $0<\rho\le R/2$, we have
\[
\int_{U^{-}_{\rho}(Y)} \abs{\vec u-\vec u_{Y,\rho}}^2 \le C \rho^{n+2+2\mu}\left(R^{-2}\norm{\vec u}^2_{\sL_{p,q}(U^{-}_{R}(X_0))}+R^{-(n+2+2\mu)} \norm{\vec u}_{\sL_2(U^{-}_R(X_0))}^2 \right).
\]
By Campanato's characterization of H\"older continuity, we find from the above inequality that
\[
[\vec u]_{\mu,\mu/2; U^{-}_{R/2}(X_0)}\le C\left(R^{-1}\norm{\vec u}_{\sL_{p,q}(U^{-}_{R}(X_0))}+R^{-(n+2+2\mu)/2} \norm{\vec u}_{\sL_2(U^{-}_R(X_0))} \right)=:CH.
\]
Then, for any $Y$ and $Y'$ in $U^{-}_{R/2}(X_0)$, we have 
\[
\abs{\vec u(Y)}\le \abs{\vec u(Y')}+\abs{\vec u(Y)-\vec u(Y')}\le \abs{\vec u(Y')}+CR^\mu H.
\]
By taking average over $Y'\in U^{-}_{R/2}(X_0)$ in the above, we get
\begin{align*}
\sup_{Y\in U^{-}_{R/2}(X_0)}\abs{\vec u(Y)}  &\le \fint_{U^{-}_{R/2}(X_0)}\abs{\vec u(Y')}\,dY'+CR^\mu H\\
&\le CR^{\mu-1}\norm{\vec u}_{\sL_{p,q}(U^{-}_R(X_0))}+C R^{-(n+2)/2} \norm{\vec u}_{\sL_2(U^{-}_R(X_0))}.
\end{align*}
Note that by H\"older's inequality, we have
\[
\norm{\vec u}_{\sL_{p,q}(U^{-}_R(X_0))} \le R^{2/q-1/p} \norm{\vec u}_{\sL_\infty(U^{-}_R(X_0))}^{(p-2)/p}
\norm{\vec u}_{\sL_2(U^{-}_R(X_0))}^{1/p}.
\]
Therefore, by combining the above two inequalities, we get
\begin{align*}
\norm{\vec u}_{\sL_\infty(U^{-}_{R/2}(X_0))} &\le CR^{-(n+2)/p}\norm{\vec u}_{\sL_\infty(U^{-}_R(X_0))}^{(p-2)/p} \norm{\vec u}_{\sL_2(U^{-}_R(X_0))}^{2/p}+C R^{-(n+2)/2} \norm{\vec u}_{\sL_2(U^{-}_R(X_0))}\\
&\le \epsilon \norm{\vec u}_{\sL_\infty(U^{-}_R(X_0))}+ C(\epsilon) R^{-(n+2)/2} \norm{\vec u}_{\sL_2(U^{-}_R(X_0))},
\end{align*}
where we used Young's inequality in the second inequality.
By a standard iteration method (see \cite[Lemma 5.1]{Gi93}), we derive \eqref{30eq2a} from the above inequality.
The proof is complete.
\hfill\qedsymbol

\subsubsection{Proof of Case (iii)}
The following lemma is an elliptic version of Lemma~\ref{lem2.3pe}, the proof of which will be given in Section~\ref{app}.

\begin{lemma}		\label{er.lem2}
Let $\Omega$ be a Lipschitz domain and let $\vec u$ be a weak solution of
\begin{equation}		\label{er.eq1}
\left\{
\begin{aligned}		
-D_\alpha(A^{\alpha\beta}(x)D_\beta \vec u)=\vec f &\quad \text{in }\, \Omega,\\
\partial \vec u/\partial \nu+\theta(x) \vec u=0 &\quad \text{on }\, \partial \Omega,
\end{aligned}
\right.
\end{equation}
where $\theta \in \sL_\infty(\partial\Omega)^{m\times m}$ and $\vec f \in L^n(\Omega)$.
Then, there exist constants $p_0>2$ and $r_0>0$ such that for all $x_0\in \overline \Omega$ and $0<r<r_0$, we have 
\begin{multline}		\label{a1}
r^{-n/p_0}\left(\int_{\Omega_{r/2}(x_0)}\abs{D\vec u}^{p_0}+\abs{\vec u}^{p_0}\,dx\right)^{1/p_0}\le C r^{-n/2}\left(\int_{\Omega_{r}(x_0)}\abs{D\vec u}^2+\abs{\vec u}^2\,dx\right)^{1/2}\\+C \norm{\vec f}_{L^n(\Omega)},
\end{multline}
where $C>0$ is a constant depending only on $m, n, \lambda, \Omega$, and $\norm{\theta}_\infty$.
\end{lemma}

Next lemma is a variant of \cite[Lemma~4.2]{Kim} and its proof is deferred to Section~\ref{app}.
\begin{lemma}		\label{lem5.6er}
Assume the condition (H1).
Let $\Omega$ be a Lipschitz domain and $\theta \in \sL_\infty(S)^{m\times m}$.
If $\vec u$ is a weak solution of 
\begin{equation}		\label{eq5.38pvd}
\left\{
\begin{aligned}		
\vec u_t-D_\alpha(A^{\alpha\beta}(x) D_\beta \vec u)=0 &\quad \text{in }\, Q=\Omega\times(a,b),\\
\partial \vec u/\partial \nu+\theta(x) \vec u=0 &\quad \text{on }\, S=\partial\Omega\times (a,b),
\end{aligned}
\right.
\end{equation}
then, the following estimates hold for all $0<r<\min(\sqrt{b-a}, \diam\Omega)$:
\begin{align}	
\label{b9}
\sup_{b-(r/2)^2\le s \le b}\norm{\vec u(\cdot, s)}_{L^2(\Omega)}&\le C  r^{-1}\norm{\vec u}_{\sL_2(\Omega\times (b-r^2,b))},\\	
\label{deq1}
\sup_{b-(r/2)^2\le s \le b} \norm{D\vec u(\cdot,s)}_{L^2(\Omega)}&\le Cr^{-2}\norm{\vec u}_{\sL_2(\Omega\times (b-r^2,b))},\\
\label{jk6}
\sup_{b-(r/2)^2\le s \le b} \norm{\vec u_t(\cdot,s)}_{L^2(\Omega)}&\le Cr^{-3}\norm{\vec u}_{\sL_2(\Omega\times (b-r^2,b))}.
\end{align}
Here, $C$ is a constant depending only on $m, n, \lambda, \Omega$, and $\norm{\theta}_\infty$.
\end{lemma}
Recall the notations \eqref{lcd}.
Below, we shall also use the notation
\[
U^{-}_{r,s}(X)=\set{(y,s)\in U^{-}_r(X)}.
\]
We see from Lemma \ref{er.lem2} that there exist constants $r_0>0$ and $p>2$ such that for any $Y\in U^{-}_{R/8}(X_0)$ with $R \le r_0$ and $r<R/8$, we have 
\begin{align}
\nonumber
\int_{U^{-}_{r,s}(Y)}\abs{D\vec u}^2 \,dx &\le Cr^{2-4/p}\left(\int_{U^{-}_{r,s}(Y)}\abs{D\vec u}^p \,dx \right)^{2/p}\le  Cr^{2-4/p}\left(\int_{U^{-}_{R/8,s}(Y)}\abs{D\vec u}^p \,dx\right)^{2/p}\\
\nonumber
&\le C\left(\frac{r}{R}\right)^{2-4/p}\int_{U^{-}_{R/4,s}(Y)} \abs{D\vec u}^2+\abs{\vec u}^2\, dx+Cr^{2-4/p}R^{4/p}\int_{\Omega\times \set{s}}\abs{\vec u_t}^2\,dx\\
\nonumber
&\le C\left(\frac{r}{R}\right)^{2-4/p}\int_{\Omega\times \set{s}} \abs{D\vec u}^2+\abs{\vec u}^2+ R^2 \abs{\vec u_t}^2\,dx.
\end{align}
Therefore, by using Lemma~\ref{lem5.6er}, we get
\[	
\int_{U^{-}_{r,s}(Y)}\abs{D\vec u}^2\,dx \le C\left(\frac{r}{R}\right)^{2-4/p}R^{-4}\int_{\Omega\times (b-R^2,b)}\abs{\vec u}^2\,dX,
\]
where we used that $R \le r_0$.
Then, we have
\[
\int_{U^{-}_r(Y)}\abs{D\vec u}^2 \, dX \le C r^{2+2\mu} R^{-4-2\mu}\int_{\Omega\times(b-R^2,b)}\abs{\vec u}^2\,dX,
\]
where $\mu=1-2/p>0$.
The above inequality corresponds to \eqref{eq5.34hungry} in the proof of Case (ii).
By repeating essentially the same argument as in the proof of Case (ii), we derive from the above inequality that
\[
\sup_{U^{-}_{R/8}(X_0)}\,\abs{\vec u}\le CR^{-2}\norm{\vec u}_{\sL_2(\Omega\times (-R^2,0))}. 
\]
By a standard covering argument, we derive \eqref{30eq2a} from the above inequality.
The proof is complete.
\hfill\qedsymbol

\subsection{Proof of Theorem~\ref{thm4.4a}}
Throughout the proof, we set
\[
\sL=\partial_t + L\quad\text{and}\quad \sL^{*}=-\partial t +L^*.
\]
We recall that (H2') implies (H2) in this setting; see \cite[Theorem~3.3]{Kim}.
Also, it is clear that  (H1') implies (H1).
We set
\[
\vec K(x,y,t)=\vec \cG(x,t,y,0),
\]
where $\vec \cG(x,t,y,s)$ is the Green's function for \eqref{RP}.

\begin{lemma}
For any $x,y\in \Omega$ with $x\neq y$, we have
$\int_0^\infty \abs{\vec K(x,y,t)}\,dt<\infty$.
\end{lemma}
\begin{proof}
Let $\vec u(x,t)= \vec K_{\cdot k}(x,y,t)$, where $k=1,\ldots, m$.
Note that $\vec u$ is a weak solution in $\sV_2^{1,0}(\Omega\times (0,\infty))^m$ of
\[
\vec u_t+ L \vec u = 0\;\text{ in }\; \Omega\times (0,\infty),\quad \partial \vec u/\partial  \nu+\Theta \vec u =0\;\text{ on }\;\partial\Omega\times (0,\infty).
\]
By a standard approximation involving Steklov average (see proof of Lemma~\ref{B3.lem1}), for a.e. $t>0$, we get
\[
\frac{d}{dt} \int_\Omega  \frac{1}{2} \abs{\vec u(x,t)}^2\,dx + \int_\Omega A^{\alpha\beta}(x) D_\beta \vec u(x,t) \cdot D_\alpha \vec u (x,t )\,dx +\ip{\Theta \vec u(\cdot, t), \vec u(\cdot,t)} =0.
\]
Therefore, by \eqref{2.eq4},  $I(t):=\int_\Omega \abs{\vec u(x,t)}^2\,dx$ satisfies
\begin{equation}		\label{eq5.24epf}
I'(t) \le -2 \vartheta_0 I(t)\quad\text{for a.e. $t>0$}.
\end{equation}
The rest of proof is the same as that of \cite[Lemma~3.12]{DK09}.
\end{proof}
We define the Green's function $\vec G(x,y)$ for \eqref{ERP} by
\begin{equation}
\label{eq:g01}
\vec G(x,y):=  \int_0^\infty \vec K(x,y,t)\,dt.
\end{equation}
We similarly define the Green's function $\vec G^{*}(x,y)$ for adjoint Robin problem.
Then, by \eqref{ctn.eq1} we find that (see \cite[Eq.~(3.21)]{DK09})
\[
\vec G^{*}(x,y)= \int_0^\infty \vec \cG^{*}(x,-t,y,0)\,dt=\int_0^\infty \vec K(y,x,t)^\top\,dt=\vec G(y,x)^\top,
\]
where $\vec \cG^{*}(x,t,y,s)$ is the Green's function for \eqref{RPs}.
We shall prove below that $\vec G(x,y)$ indeed enjoys the properties stated in Section~\ref{sec:4.2enf}.
Denote
\[
\tilde{\vec K}(x,y,t)=\int_0^t \vec K(x,y,s)\,ds.
\]
so that we have
\[
\vec G(x,y)=\lim_{t\to \infty} \tilde{\vec K}(x,y,t).
\]

\begin{lemma}			\label{lem4.10a}
The following holds uniformly for  $t>0$ and $y\in\Omega$, where $d_y=\dist(y,\partial\Omega)$.
\begin{enumerate}[i)]
\item
$\norm{\tilde{\vec K}(\cdot,y,t)}_{L^p(B(y,d_y))} \le C  d_y^{n/p-n}\quad \text{for }\; p \in \bigl[1,\frac{n+2}{n}\bigr)$.

\item
$\norm{\tilde{\vec K}(\cdot,y,t)}_{L^{2(n+2)/n}(\Omega\setminus B(y,r))}
\le C r^{-\frac{n(n+4)}{2(n+2)}},\quad \forall r \in (0, d_y]$.

\item
$\norm{D\tilde{\vec K}(\cdot,y,t)}_{L^p(B(y,d_y))}\le C  d_y^{-1-n+n/p} \quad \text{for }\; p \in  \bigl[1,\frac{n+2}{n+1}\bigr)$.

\item
$\norm{D\tilde{\vec K}(\cdot,y,t)}_{L^2(\Omega\setminus B(y,r))} \le C  r^{-1-n/2},\quad \forall r \in (0, d_y]$.
\end{enumerate}
In the above, $C$ is a constant depending only on $n, m, \lambda, p, \diam \Omega$, and parameters appearing in (H1') and (H2').
\end{lemma}
\begin{proof}
See \cite[Lemma~3.23]{DK09}.
\end{proof}

We need to show that $\vec G(x,y)$ defined by the formula \eqref{eq:g01} satisfies the properties i) -- iii) in Section~\ref{sec:4.2enf}.
We begin with iii).
Let $\vec u$ be defined by \eqref{eq4.01e}. 
Note that
\[
\vec v(x,t) := \int_\Omega \tilde{\vec K}(y,x,t)^\top \vec f(y)\,dy= \int_0^t \!\!\!\int_\Omega \vec K(y,x,s)^\top \vec f(y)\,dy\, ds
\]
is absolutely convergent by Lemma~\ref{lem4.10a}, and
\begin{align}
\label{eq:E29a}
\lim_{t\to\infty} \vec v(x,t)&=\int_\Omega \vec G(y,x)^\top \vec f(y)\,dy=\vec u(x),\\
\nonumber
\vec v_t(x,t)&=\int_\Omega \vec K(y,x,t)^\top \vec f(y)=\int_\Omega \vec \cG^{*}(x,-t,y,0) \vec f(y)\,dy.
\end{align}
By differential inequality \eqref{eq5.24epf} applied to $\vec v_t$, we get (c.f. \cite[Eq.~(3.43)]{DK09})
\begin{equation}		\label{eq:E29b}
\norm{\vec v_t(\cdot,t)}_{L^2(\Omega)} \le e^{-\vartheta_0 t}\norm{\vec f}_{L^2(\Omega)},\quad \forall t>0,
\end{equation}
which implies $\vec v_t(\cdot,t)\in L^2(\Omega)^m$.
By using $\vec \cG^{*}(x,t,y,s)=\vec \cG^{*}(x,t-s,y,0)$, we have
\[
\vec v(x,t)=\int_{-t}^0 \int_\Omega \vec \cG^{*}(x,-t,y,s) \vec f(y)\,dy\, ds,
\]
and thus, for a.e. $t>0$, we get (c.f. \cite[Eq.~(3.46)]{DK09})
\begin{equation}	\label{eq5.28jh}
\int_\Omega \vec v(\cdot,t)\cdot \vec v_t(\cdot,t)+\int_\Omega A^{\alpha\beta}D_\beta \vec v(\cdot,t) \cdot D_\alpha \vec v(\cdot,t)+\ip{\Theta\vec v(\cdot,t),\vec v(\cdot,t)}
=\int_\Omega \vec f \cdot \vec v(\cdot,t).
\end{equation}
We  record that $\vec v$ also satisfies for a.e. $t>0$ the identity
\begin{equation}	\label{eq5.29td}
\int_\Omega \vec v_t(\cdot, t)\cdot \vec \phi +\int_\Omega A^{\alpha\beta}D_\beta \vec \phi \cdot D_\alpha \vec v(\cdot,t)+\ip{\Theta\vec \phi,\vec v(\cdot,t)}
=\int_\Omega \vec f\cdot \vec \phi,\quad \forall \vec \phi \in H^1(\Omega)^m.
\end{equation}
By using \eqref{eq:E29b}, we obtain from \eqref{eq5.28jh} that
\[
\tilde{\lambda} \norm{D \vec v(\cdot,t)}_{L^2(\Omega)}+\ip{\Theta \vec v(\cdot,t), \vec v(\cdot,t)} \le C \norm{\vec f}_{L^2(\Omega)} \norm{\vec v}_{L^2(\Omega)}
\]
Therefore, by \eqref{2.eq4} and Cauchy's inequality, for a.e. $t>0$, we have
\[
\norm{\vec v(\cdot,t )}_{H^{1}(\Omega)} \le C \norm{\vec f}_{L^2(\Omega)}.
\]
Then, there is a sequence $t_m\to \infty$ such that $\vec v(\cdot, t_m) \rightharpoonup \tilde{\vec u}$ weakly in $H^{1}(\Omega)^m$ for some $\tilde{\vec u} \in H^{1}(\Omega)^m$.
By \eqref{eq:E29a}, we must have $\vec u=\tilde{\vec u}$.
Then, by using  \eqref{eq:E29b} and taking limit in \eqref{eq5.29td}, we see that $\vec u$ is a weak solution of the problem \eqref{eq4.02e}.
We have thus verified the property iii).
By repeating essentially the same proof of \cite[Theorem~2.12]{DK09}, we find that $\vec G(x,y)$ satisfies the properties i) and ii) as well.

Now, suppose (H3') also holds.
In the rest of the proof, we shall denote
\[
R:= \diam\Omega.
\]
It is clear that  (H3') implies (H3) and thus, by \eqref{3.eq2}, for $X=(x,t)\in Q:=\Omega\times (-\infty,\infty)$ satisfying $\abs{t}\le R^2$, we have
\begin{equation}					\label{eq618z}
\abs{\vec K(x,y,t)} \le C\abs{X-\bar{Y}}_{\sP}^{-n},\quad \bar Y=(y,0).
\end{equation}
By using \eqref{app.eq1c} and (H3'), we get similar to \eqref{wse.eq2b} that
\[
\tri{\vec K(\cdot,y)}_{Q \setminus Q(\bar{Y},r)} \le Cr^{-n/2},\quad   \forall r \in (0,R].
\]
Similar to \cite[Eq.~(3.59)]{DK09}, for $0<r \le R$ and $t \ge 2r^2$, we have
\begin{equation}							\label{eq12.27g}
\abs{\vec K(x,y,t)} \le C r^{-n}e^{-\vartheta_0 (t-2r^2)}.
\end{equation}
By \eqref{eq:g01} we have
\[
\abs{\vec G(x,y)} \le \int_0^{\abs{x-y}^2} + \int_{\abs{x-y}^2}^{2R^2}+\int_{2 R^2}^\infty \abs{\vec K(x,y,t)}\,dt =: I_1+I_2+I_3.
\]
It then follows from \eqref{eq618z} and \eqref{eq12.27g} that
\begin{align*}
I_1 &\le C\int_0^{\abs{x-y}^2}\abs{x-y}^{-n}\,dt \le C\abs{x-y}^{2-n},\\
I_2&\le C\int_{\abs{x-y}^2}^{2R^2}t^{-n/2}\,dt\le
\begin{cases}  C+C\ln (R/\abs{x-y}) & \text{if $n=2$,}
\\
C\abs{x-y}^{2-n} &\text{if $n\ge 3$.}
\end{cases} \\
I_3&\le C\int_{2R^2}^\infty d^{-n} e^{-\vartheta_0 (t-2R^2)}\,dt \le C \vartheta_0^{-1} R^{-n}.
\end{align*}
Combining all together we get that if $0<\abs{x-y}\le r$, then
\[
\abs{\vec G(x,y)} \le
\begin{cases}  C\left(1+\vartheta_0^{-1} R^{-2}+ \ln (R/\abs{x-y}) \right) & \text{if $n=2$,}\\
C\left(1+ \vartheta_0^{-1} R^{-2}\right)\abs{x-y}^{2-n} &\text{if $n\ge 3$.}
\end{cases}
\]
The theorem is proved.
\hfill\qedsymbol

\section{Proofs of technical lemmas}		\label{app}

\subsection{Proof of Lemma~\ref{B3.lem1}}
Let $\set{\vec \psi_k}_{k=1}^\infty$ be an orthogonal basis for $H^1(\Omega)^m$ that is normalized so that
\[
(\vec \psi_k, \vec \psi_l):=\int_\Omega \vec \psi_k \cdot \vec \psi_l=\delta_{kl}.
\]
By standard Galerkin's method, we construct an approximate solution $\vec u^N(x,t)$ of the form
\[
\vec u^N(x,t)=\sum_{k=1}^{N} c^N_k(t)\vec \psi_k(x)
\]
where $c_k^N(t)=(\vec u^N(\cdot,t),\vec \psi_k)$ are determined by the conditions
\[
\frac{d}{dt} (\vec u^N(\cdot,t),\vec \psi_k)+(A^{\alpha\beta}(\cdot,t)D_\beta \vec u^N(\cdot,t),D_\alpha \vec \psi_k)+\ip{\Theta(t) \vec u^N(\cdot,t),\vec \psi_k}=(\vec f(\cdot,t), \vec \psi_k)
\]
and
\[
c^N_k(a)=(\vec \psi_0,\vec \psi_k)
\]
for $k=1,\cdots, N$.
Therefore, one can easily verify that for any $a \le t_1<t_2 \le b$, we have
\[
\frac{1}{2}\Norm{\vec u^N(\cdot,t)}_{L^2(\Omega)}^2\, \Big\vert_{t=t_1}^{t=t_2}+\int_{t_1}^{t_2}\!\!\! \int_\Omega A^{\alpha\beta}D_\beta \vec u^N \cdot D_\alpha \vec u^N+\int_{t_1}^{t_2}\Ip{\Theta \vec u^N,\vec u^N}=\int_{t_1}^{t_2}\!\!\!\int_\Omega \vec f \cdot \vec u^N,
\]
and thus, by \cite[Lemma~2.1, p.139]{LSU} and \eqref{eq2.19ar} we have the uniform estimate
\[
\tri{\vec u^N}_{Q}  \le C \tri{\vec u^N}_{\Theta;Q}\le C\left(\norm{\vec f}_{\sL_{2,1}(Q)}+ \norm{\vec \psi_0}_{L^2(\Omega)}\right) <\infty
\]
for all $N=1,2,\ldots$.
Then by following literally the same steps as in the proof of \cite[Theorem~4.1, p. 153]{LSU}, we find that there exists a weak solution $\vec u \in \sV_2(Q)^m$ of the problem \eqref{B2.eq2}. 

Next, we show that $\vec u \in \sV_2(Q)^m$ obtained above actually belongs to $\sV^{1,0}_2(Q)^m$;
i.e. $\vec u$ is strongly continuous in $t$ in the norm of $L^2(\Omega)^m$.
To see this, we follow the argument in \cite[pp. 156-159]{LSU}.
Without loss of generality, we assume that $[a,b]=[0,T]$.
Note that $\vec u$ satisfies the identity
\[
-\int_Q \vec u \cdot \vec \phi_t\,dX-\int_\Omega \vec \psi_0 \cdot \vec \phi(\cdot,0)\,dx=\int_Q \vec F_\alpha \cdot D_\alpha \vec \phi\,dX-\int_0^T \ip{\Theta \vec u,\vec \phi}\,dt+\int_Q \vec f\cdot \vec \phi\,dX
\]
for any $\vec \phi\in \sW^{1,1}_2(Q)^m$ that is equal to zero for $t=T$.
Here, we set
\[
\vec F_\alpha=-A^{\alpha\beta}D_\beta \vec u.
\]
Let $\tilde{\vec u}$ be an even extension of $\vec u$ and let $\tilde{\vec F}_\alpha, \tilde{\Theta}$, and $\tilde{\vec f}$, respectively, be odd extensions of $\vec F_\alpha, \Theta$, and $\vec f$ onto $\tilde{Q}=\Omega\times (-\infty,\infty)$ similar to \cite[Eq.~(4.9), p. 157]{LSU}
so that we have the identity
\begin{equation}		\label{c1}
-\int_{\tilde{Q}} \tilde{\vec u} \cdot \vec \phi_t=\int_{\tilde{Q}} \tilde{\vec F}_\alpha \cdot D_\alpha \vec \phi-\int_{-\infty}^\infty \Ip{\tilde{\Theta} \tilde{\vec u},\vec \phi}+\int_{\tilde{Q}} \tilde{\vec f}\cdot \vec \phi
\end{equation}
for any $\vec \phi\in \sW^{1,1}_2(\tilde{Q})^m$ that is equal to zero for $\abs{t}\ge T$.
We note that the identity \eqref{c1} corresponds to \cite[Eq.~(4.10), p. 157]{LSU} and we derive from \eqref{c1} the following identity that corresponds to \cite[Eq.~(4.12), p. 157]{LSU}:
\[
-\int_{\tilde{Q}} \vec v \cdot \vec \Phi_t =\int_{\tilde{Q}} \left( \vec G^\alpha \cdot D_\alpha \vec \Phi+  \vec g \cdot \vec \Phi \right) -\int_{-\infty}^\infty\Ip{\tilde{\Theta} \vec v,\vec \Phi},
\]
where $\vec v$, $\vec G^\alpha$ and $\vec g$ are given element of spaces $\sV_2(\tilde{Q})^m$,  $\sL_2(\tilde{Q})^m$ and $\sL_{2,1}(\tilde{Q})^m$, respectively, that are equal to zero for $\abs{t} \ge T$, while $\vec \Phi$ is an arbitrary element of $\sW^{1,1}_2(\tilde{Q})^m$.
As a matter of fact, we have
\[
\vec v(x,t)=\omega(t) \tilde{\vec u}(x,t), \quad \vec G^\alpha=\omega\tilde{\vec F}_\alpha, \quad \text{and}\quad \vec g=\omega\tilde{\vec f} +\omega_t \tilde{\vec u},
\]
where $\omega$ is a smooth function that is equal to $1$ for $\abs{t} \le T-\delta$, for some positive number $\delta$, and to zero for $\abs{t}\ge T$.
Similar to \cite[Eq.~(4.17), p. 159]{LSU}, we get
\begin{multline*}
\frac{1}{2}\norm{\vec v_{h_1}-\vec v_{h_2}}^2_{L^2(\Omega)}  \Big|_{t=t_1}^{t=t_2} =
\int_{t_1}^{t_2}\!\!\!\int_{\Omega}\left\{(\vec G^\alpha_{h_1}-\vec G^\alpha_{h_2}) \cdot D_\alpha (\vec v_{h_1}-\vec v_{h_2}) +(\vec g_{h_1}-\vec g_{h_2})\cdot (\vec v_{h_1}-\vec v_{h_2})\right\}\\
-\int_{t_1}^{t_2}\left( \Ip{\tilde{\Theta}\vec v,\vec v_{h_1}-\vec v_{h_2}}_{h_1}-\Ip{\tilde{\Theta} \vec v,\vec v_{h_1}-\vec v_{h_2}}_{h_2} \right)dt,
\end{multline*}
where we used the notation $\vec u_h(x,t)=\fint_t^{t+h} \vec u(x,\tau)\,d\tau$.
Note that we have
\begin{multline}			\label{eq6.02bk}
\Abs{\int_{t_1}^{t_2}\Ip{\tilde{\Theta}\vec v,\vec v_{h_1}-\vec v_{h_2}}_{h}\,dt }\le\int_{-\infty}^\infty \Abs{\Ip{\tilde{\Theta}(\tau)\vec v(\cdot,\tau),\vec v_{h_1}(\cdot,\tau)-\vec v_{h_2}(\cdot,\tau)}}\,d\tau\\
\le C \left(\int_{-\infty}^\infty \norm{\vec v(\cdot,\tau)}^2_{H^1(\Omega)}\,d\tau\right)^{1/2}\left(\int_{-\infty}^\infty \norm{\vec v_{h_1}(\cdot,\tau)-\vec v_{h_2}(\cdot,\tau)}^2_{H^1(\Omega)}\,d\tau\right)^{1/2}
\end{multline}
uniformly for all $t_1$, $t_2$ and $h$.
Since $\vec v \in \sV_2(\tilde{Q})^m$ and vanishes for $\abs{t}\ge T$, the last term in the above inequality tends to zero as $h_1$ and  $h_2$ tends to zero.
Therefore, arguing similar to the proof of \cite[Lemma~4.1, p. 158]{LSU}, we find that $\vec u \in \sV^{1,0}_2(Q)^m$.

Finally, we prove the energy inequality \eqref{enieq}.
The uniqueness of weak solution in the space $\sV^{1,0}_2(Q)^m$ is a simple consequence of \eqref{enieq} and \eqref{eq2.19ar}.
The assumptions on $\Theta$ in (H1) implies that for a.e. $t\in(-\infty, \infty)$, we have (see \cite[Chapter V, \S 5]{Yosida})
\begin{equation}	\label{6.02sf}
\Ip{\Theta(t)\vec u, \int_\alpha^\beta \vec v(\cdot, \tau)\,d\tau}=\int_{\alpha}^\beta \ip{\Theta(t) \vec u,\vec v(\cdot, \tau)}\,d\tau
\end{equation}
for any $\vec u \in H^1(\Omega)^m$ and $\vec v \in \sW^{1,0}_2( \Omega\times (\alpha,\beta))^m$.
Then, similar to \cite[Eq.~(2.12), p. 142]{LSU}, for any $t_1 \in [0,T-h]$, we have the identity
\begin{multline}		\label{jj.eq1b}
\frac{1}{2}\int_\Omega \abs{\vec u_h}^2(x,t)\,dx\;\Big|_{t=0}^{t=t_1}+\int_0^{t_1}\!\!\!\int_\Omega (A^{\alpha\beta}D_\beta \vec u)_h \cdot D_\alpha \vec u_h \, dX\\
+\int_0^{t_1}\!\!\! \fint_{\tau}^{\tau+h} \ip{\Theta(t)\vec u(\cdot,t),\vec u_h(\cdot,\tau)}\,dt \,d\tau=\int_0^{t_1}\!\!\!\int_\Omega \vec f_h\cdot  \vec u_h.
\end{multline}
Note that by Fubini's theorem and \eqref{6.02sf}, we have
\begin{align*}
\int_0^{t_1}\!\!\! \fint_{\tau}^{\tau+h} \ip{\Theta(t)\vec u(\cdot,t),\vec u_h(\cdot,\tau)}\,dt \,d\tau
&=\int_{-\infty}^\infty \Ip{\Theta(t)\vec u(\cdot,t), \int_0^{t_1} \frac{1}{h} 1_{(t-h,t)}(\tau) \vec u_h(\cdot, \tau)\,d\tau}dt \\
&=\int_{-\infty}^\infty \Ip{\Theta(t)\vec u(\cdot,t), \int_{-\infty}^{\infty} \chi_h(s-t) 1_{(0,t_1)}(s) \vec u(\cdot, s)\,ds}dt,
\end{align*}
where $\chi_h(t)=h^{-1}(1-\abs{x/h})_{+}$.
Therefore, by setting $\vec v(x,t)= 1_{(0,t_1)}(t) \vec u(x,t)$, we have
\begin{multline*}
\int_0^{t_1}\!\!\! \fint_{\tau}^{\tau+h} \ip{\Theta(t)\vec u(\cdot,t),\vec u_h(\cdot,\tau)}\,dt \,d\tau-\int_0^{t_1}\ip{\Theta(t)\vec u(\cdot,t),\vec u(\cdot,t)}\,dt\\
=\int_{-\infty}^\infty \Ip{\Theta(t)\vec u(\cdot,t), \int_{-\infty}^{\infty} \chi_h(t-s)  \vec v(\cdot, s)\,ds -\vec v(\cdot,t) }\,dt=:I(h).
\end{multline*}
Similar to \eqref{eq6.02bk}, we have
\[
\abs{I(h)} \le C \left(\int_{-h}^{t_1+h} \norm{\vec u(\cdot,t)}^2_{H^1(\Omega)}\,dt \right)^{1/2}\left(\int_{-h}^{t_1+h} \norm{\chi_h * \vec v(\cdot,t)-\vec v(\cdot,t)}^2_{H^1(\Omega)}\,dt\right)^{1/2}\to 0
\]
as $h$ tends to zero.
Therefore, from \eqref{jj.eq1b}, we obtain
\[
\frac{1}{2}\norm{\vec u(\cdot,t)}_{L^2(\Omega)}^2\,  \Big|_{t=0}^{t=t_1}+\int_0^{t_1}\!\!\!\int_\Omega A^{\alpha\beta}D_\beta \vec u \cdot D_\alpha \vec u+\int_0^{t_1}\ip{\Theta(t)\vec u(\cdot,t),\vec u(\cdot,t)}=\int_0^{t_1}\!\!\!\int_\Omega \vec f \cdot \vec u.
\]
By using \eqref{eq2.14ar} and Cauchy's inequality, we get  \eqref{enieq} from the above identity.
\hfill\qedsymbol

\subsection{Proof of Lemma~\ref{lem2.2pe}}
By \cite[Lemma~5.3]{GM09}, we see that $\Theta=M_\theta$ satisfies \eqref{eq2.2p} and \eqref{eq2.3p}.
Therefore, we only need to establish the inequality \eqref{2.eq4}.
For this purpose, we first show that there exists $\alpha\in \bR$ such that for a.e. $t\in (-\infty,\infty)$, we have
\begin{equation}		\label{rz4aa}
0\le \alpha \int_\Omega \abs{D\vec u}^2\,dx +\int_{\partial\Omega} \theta(\cdot,t) \vec u \cdot \vec u\,dS_x,\quad \forall \vec u \in H^1(\Omega)^m.
\end{equation}
To see \eqref{rz4aa}, it is enough to show that for a.e. $t\in (-\infty,\infty)$, we have
\[
\inf\Set{\int_{\partial \Omega} \theta(\cdot,t) \vec u\cdot \vec u\,dS_x : \vec u\in H^1(\Omega)^m, \; \norm{D\vec u}_{L^2(\Omega)}=1}>-\infty.
\]
Indeed, for any $\vec u\in H^1(\Omega)^m$ satisfying $\norm{D\vec u}_{L^2(\Omega)}=1$, we write
\[
\vec u=\left(\vec u-\fint_{\partial \Omega}\vec u\right)+\fint_{\partial \Omega}\vec u=:\vec v+\vec c.
\]
Note that there exists a constant $\beta=\beta(n, p, \Omega)$ such that
\begin{equation}		\label{rzz1}
\norm{\vec v}_{L^{2p/(p-1)}(\partial\Omega)}\le  \beta \norm{D \vec v}_{L^2(\Omega)}=\beta,
\end{equation}
where $p$ is as in \eqref{eq2.09ex} and we used $\int_{\partial\Omega} \vec v=0$.
Then by \eqref{eq2.10ex} and \eqref{rzz1}, we have
\begin{align*}
\int_{\partial \Omega} \theta \vec u\cdot \vec u \,dS_x & \ge \int_{\partial \Omega} \theta \vec v \cdot \vec v \,dS_x+ \left( \int_{\partial \Omega} (\theta+\theta^{*})\vec v \,dS_x \right) \cdot \vec c + \delta \abs{\vec c}^2\\
&\ge \int_{\partial \Omega} \theta \vec v \cdot \vec v\,dS_x -\frac{1}{4\delta} \Abs{\int_{\partial \Omega} (\theta+\theta^{*})\vec v \,dS_x}^2\\
&\ge -\norm{\theta}_{\sL_{p,\infty}} \norm{\vec v}^2_{L^{2p/(p-1)}(\partial \Omega)} -\delta^{-1} \norm{\theta}_{\sL_{p,\infty}}^2 \norm{\vec v}^2_{L^{2p/(p-1)}(\partial \Omega)}\abs{\partial\Omega}^{1-1/p}\\
& \ge - \beta^2\left(\norm{\theta}_{\sL_{p,\infty}}+\delta^{-1} \abs{\partial\Omega}^{1-1/p}\norm{\theta}_{\sL_{p,\infty}}^2 \right).
\end{align*}
Therefore, we get the inequality \eqref{rz4aa} with
\begin{equation}	\label{eq6.05apx}
\alpha \le \beta^2\left(\norm{\theta}_{\sL_{p,\infty}}+\delta^{-1} \abs{\partial\Omega}^{1-1/p}\norm{\theta}_{\sL_{p,\infty}}^2 \right)=:\Lambda.
\end{equation}
Next, we claim that for any $\epsilon>0$, there exists $C_\epsilon>0$ such that
\begin{equation}	\label{eq6.06apx}
\norm{\vec u}_{L^2(\Omega)}^2 \le C_\epsilon \left( (\Lambda+\epsilon) \int_\Omega \abs{D\vec u}^2\,dx +\int_{\partial\Omega} \theta(\cdot,t) \vec u \cdot \vec u\,dS_x\right).
\end{equation}
Note that once we establish \eqref{eq6.06apx}, then by \eqref{rz4aa} we would have
\[
\norm{\vec u}_{H^1(\Omega)}^2 \le C_\epsilon' \left( (\Lambda+\epsilon) \int_\Omega \abs{D\vec u}^2\,dx +\int_{\partial\Omega} \theta(\cdot,t) \vec u \cdot \vec u\,dS_x\right)
\]
for some $C_\epsilon'>0$ and the lemma follows by \eqref{eq6.05apx} or nonnegative definiteness of $\theta$.

Finally, we prove \eqref{eq6.06apx} by a usual contradiction argument.
If the stated estimate were false, for each positive integer $k$, there would exist a function $\vec u_k \in H^1(\Omega)^m$ such that $\norm{\vec u_k}_{L^2(\Omega)}=1$ and
\[
\epsilon \int_\Omega \abs{D\vec u_k}^2\,dx \le (\Lambda+\epsilon) \int_\Omega \abs{D\vec u_k}^2\,dx +\int_{\partial\Omega} \theta(\cdot,t) \vec u_k \cdot \vec u_k\,dS_x\ \le \frac{1}{k},
\]
where we used \eqref{rz4aa}.
Therefore, we would have
\begin{equation}		\label{rzz1a}
\lim_{k\to \infty} \int_\Omega \abs{D\vec u_k}^2\,dx= 0=\lim_{k\to \infty} \int_{\partial\Omega} \theta(\cdot,t) \vec u_k \cdot \vec u_k\,dS_x.
\end{equation}
Then by Rellich-Kondrachov compactness theorem, there exists $\vec u \in L^2(\Omega)^m$ such that (by passing to a subsequence) $\vec u_k\to \vec u$ in $L^2(\Omega)^m$. 
Also, it follows from \eqref{rzz1a} that $D \vec u_k \to 0$ in $L^2(\Omega)^m$.
Therefore, we have
\begin{equation}		\label{rz3}
\norm{\vec u}_{L^2(\Omega)}=1, \quad D\vec u= 0, \quad \text{and }\, \vec u_k \to \vec u \quad \text{in }\, H^1(\Omega)^m.
\end{equation}
Then, by combining \eqref{rzz1a} and \eqref{rz3}, we conclude that $\vec u$ is constant and
\[
\int_{\partial\Omega} \theta(\cdot,t) \vec u \cdot \vec u\,dS_x=0,
\]
which contradicts \eqref{eq2.10ex}.
\hfill\qedsymbol

\subsection{Proof of Lemma~\ref{lem2.3pe}}
Without loss of generality, we assume that $a=-T$ and $b=0$ so that $Q=\Omega\times (-T,0)$ and $S=\partial\Omega \times (-T,0)$.
The proof is based on establishing a reverse H\"older inequality and applying a variant of Gehring's lemma as presented in Arkhipova~\cite{Arkhipova99}.
Below, we denote
\[
\delta(X, \partial \bQ)=\inf \set{ \abs{X-Y}_\sP : Y\in \partial \bQ}.
\]
\begin{theorem}[{\cite[Theorem~1]{Arkhipova99}}]		\label{thm_Ark}
Denote by $D_r(x)=\set{y\in \bR^n : \abs{y_i-x_i} <r, i=1,\ldots, n}$ and $\bQ=D_{3/2}(0) \times (-5/4,5/4)$.
Let $g\in \sL_p(\bQ)$, $p>1$, be a nonnegative function.
Suppose that for some $R_0>0$ and for all $X \in \bQ$ the following inequality holds for all $R \le \hat{R}_a=\frac{1}{a} \min\set{R_0, \delta(X, \partial \bQ)}$, where $\delta(X, \partial \bQ)=\inf \set{ \abs{X-Y}_\sP : Y\in \partial \bQ}$:
\begin{equation}		\label{eq6.10unscaled}
\fint_{Q_R(X)} g^p + \psi_R(X) \le \epsilon \left(\fint_{Q_{a R}(X)} g^p + \psi_{a R}(X) \right) + B\left(\fint_{Q_{\alpha R}(X)} g \right)^p,
\end{equation}
where $\epsilon \in (0,1)$, $a>1$, $B>1$,
and $\psi_\rho(X)$ is a nonnegative function defined for all $X \in \bQ$ and $\rho>0$ such that $\sup_\rho \psi_\rho(X) <\infty$ for a.e. $X \in \bQ$ and 
\[
\esssup_{X \in \bQ} \sup_{\set{\rho: Q_\rho(X)\subset \bQ}} \psi_\rho(X) \rho^{n+2}=:m_\psi(\bQ)<\infty.
\]
Then there exist constants $p_0>p$ and $c_0>0$ such that $g \in \sL_{p_0}(\bQ')$ for all $\overline{\bQ'} \subset \bQ$ and
\[
\norm{g}_{\sL_{p_0}(\bQ')} \le c_0 \left\{ \norm{g}_{\sL_p(\bQ)} + m_\psi^{1/p}(\bQ)\right\}.
\]
The constants $p_0$ and $c_0$ depend on $p, n, B, a$, and $\delta$.
In addition, $c_0$ depends on $\dist(\bQ', \partial \bQ)$.
\end{theorem}

One can reformulate Theorem~\ref{thm_Ark} with the cube $\bQ$ replaced by a cube
\[
\bP_r(X_0):=D_r(x_0) \times (t_0-r^2, t_0+r^2).
\]
Indeed, by scaling, one can see that if inequality \eqref{eq6.10unscaled} holds for all $X\in \bP_r(X_0)$ and $R \le \frac{1}{a} \min \set{r R_0, \delta(X, \partial \bP_r)}$, then
we have
\begin{equation}		\label{rvhscaled}
\left(\fint_{\bP_{r/2}(X_0)} \abs{g}^{p_0}\right)^{1/p_0}  \le c_0 \left\{ \left(\fint_{\bP_r(X_0)} \abs{g}^p\right)^{1/p} + r^{-(n+2)/p} m_\psi^{1/p}(\bP_r(X_0))\right\}.
\end{equation}
Recall the notations \eqref{lcd}.
Hereafter in the proof, we also use the notation
\[
\vec u_\rho=\vec u_{X,\rho}=\fint_{U^{-}_\rho(X)} \vec u.
\]
Let $\vec u$ be a weak solution of \eqref{eqj.1b}.
We claim that for any $\epsilon \in (0,1)$, there exist $r_0>0$ and $C>0$ such that for all $X=(x,t)\in \overline \Omega \times (-T,0]$ and $0<\rho <\min( \sqrt{t+T}, r_0)$, we have 
\begin{multline}			\label{eq6.11sleepy}		
\int_{U^{-}_{\rho/2}(X)} \left(\abs{D\vec u}^2+\abs{\vec u}^2+(\rho/2)^{-2}\abs{\vec u-\vec u_{\rho/2}}^2\right)\\
\le \epsilon \int_{U^{-}_\rho(X)} \left(\abs{D\vec u}^2+\abs{\vec u}^2+\rho^{-2}\abs{\vec u-\vec u_\rho}^2\right)
+C \rho^{(n+2)(1-\frac{2}{q})}\left(\int_{U^{-}_\rho(X)} \abs{D\vec u}^q+\abs{\vec u}^q\right)^{\frac{2}{q}},
\end{multline}
where $q=2n/(n+2)$ if $n\ge 3$ and $1<q<3/2$ if $n=2$.

Take the claim for now.
For $x_0 \in \Omega$, write $\hat{X}_0=(x_0,0)$ and consider $\bP_r=\bP_r(\hat{X}_0)$, where $r< \min(\sqrt{T}, r_0)$ is fixed.
We define
\[
g(X)=\left(\abs{D\vec u}^2+\abs{\vec u}^2\right)^{\frac{q}{2}}(X) \quad \text{for }\, X\in Q \cap \bP_r
\]
and extend it by zero on $\bP_r \setminus Q$.
Also, for any $X\in \bR^{n+1}$ and $\rho>0$, we define
\[
\psi_\rho(X)=\rho^{-n-4}\int_{U_\rho(X)} \abs{\vec u - \hat{\vec u}_{X,\rho}}^2,\quad\text{where }\; \hat{\vec u}_{X,\rho}=\fint_{U_\rho(X)} \vec u.
\]
It should be understood that $\psi_\rho(X)=0$ if $U_\rho(X)=\emptyset$.
Note that if $U_\rho(X) \neq \emptyset$, then there exist $X'=(x',t')$ such that $\abs{X-X'}_\sP <\rho$ and $X'\in Q$.
If we denote $Y=(x', s)$, where $s=\min(t+\rho^2, 0)$, then $Y \in \overline{U_\rho(X)}$ and $U_\rho(X) \subset U_{2\rho}^-(Y)$, and thus we have
\[
\psi_\rho(X) \le  4 \rho^{-n-4}\int_{U_\rho(X)}\abs{\vec u-\vec u_{Y,2\rho}}^2 
\le 4 \rho^{-n-4}\int_{U_{2\rho}^{-}(Y)}\abs{\vec u-\vec u_{Y,2\rho}}^2
\]
and since $U_{4\rho}(Y) \subset U_{5\rho}(X)$, we also have
\[
\rho^{-n-4}\int_{U^{-}_{4\rho}(Y)}\abs{\vec u-\vec u_{Y,4\rho}}^2\le 
4\rho^{-n-4}\int_{U^{-}_{4\rho}(Y)}\abs{\vec u- \hat{\vec u}_{X,5\rho}}^2\le 
4 \psi_{5\rho}(X).
\]
Therefore,  for all $X\in \bP_r$ and $\rho<\frac{1}{5} \min(\delta(X,\partial \bP_r), r)$ satisfying $U_\rho(X)\neq \emptyset$, we get from  \eqref{eq6.11sleepy}
\begin{align}
\nonumber
\fint_{Q_\rho(X)}g^{2/q}+\psi_\rho(X)&\le \frac{C}{\rho^{n+2}}\int_{U^{-}_{2\rho}(Y)}\left(g^{2/q}+(2\rho)^{-2}\abs{\vec u-\vec u_{Y,2\rho}}^2\right) \\
\nonumber
&\le \frac{C \epsilon}{\rho^{n+2}}\int_{U^{-}_{4\rho}(Y)}\left(g^{2/q}+(4\rho)^{-2} \abs{\vec u-\vec u_{Y,4\rho}}^2\right)+\frac{C}{\rho^{2(n+2)/q}}\left(\int_{U^{-}_{4\rho}(Y)}g\right)^{2/q}\\
\label{plseq1}
&\le C \epsilon \left(\fint_{Q_{5\rho}(X)}g^{2/q}+\psi_{5\rho}(X)\right)+C\left(\fint_{Q_{5\rho}(X)}g\right)^{2/q}.
\end{align}
On the other hand, if $U_\rho(X)=\emptyset$, then we have 
\begin{equation}		\label{plseq2}
\fint_{Q_\rho(X)}g^{2/q}+\psi_\rho(X)=0.
\end{equation}
By \eqref{plseq1} and \eqref{plseq2}, we get for any $X\in \bP_r$ and $\rho<\frac{1}{5}\min (\delta(X,\partial \bP_r),r)$ that
\[
\fint_{Q_\rho(X)}g^{2/q}+\psi_\rho(X)\le \delta\left(\fint_{Q_{8\rho}(X)}g^{2/q}+\psi_{8\rho}(X)\right)+C\left(\fint_{Q_{8\rho}(X)}g\right)^{2/q},
\]
for some $\epsilon \in (0,1)$.
On the other hand, note that
\[
\psi_\rho(X)\rho^{n+2}\le C\rho^{-2}\int_{U_\rho(X)}\abs{\vec u}^2 \le C\norm{\vec u}_{\sL_{2,\infty}(U_\rho(X))}^2,
\]
and thus, we have 
\[
m_\psi(\bP_r) \le C\norm{\vec u}_{\sL_{2,\infty}(\bP_r(\hat{X_0})\cap Q)}^2.
\]
Then, we take $a=5$, $R_0=1$ and apply the scaled version of Theorem~\ref{thm_Ark} to get via \eqref{rvhscaled} that
\[
\left(\fint_{\bP_{r/2}(\hat{X}_0)}g^{p_0}\,dX\right)^{1/p_0}\le C\left(\fint_{\bP_r(\hat{X}_0)}g^{2/q}\,dX\right)^{q/2}+C\left(r^{-(n+2)}\norm{\vec u}^2_{\sL_{2,\infty}(\bP_r(\hat{X}_0) \cap Q)}\right)^{q/2},
\]
where $p_0>2/q$.
Therefore, by setting $q_0=p_0q >2$ and using a usual covering argument, we obtain \eqref{drowsy}.

It only remains to establish \eqref{eq6.11sleepy}.
Hereafter, we shall denote
\[
Q'=\Omega \times (t_0-r^2,t_0)\quad\text{and}\quad S'=\partial\Omega \times (t_0-r^2,t_0).
\]
Fix $\kappa \in (0,1)$ and a function $\tau \in C^\infty_c(\bR)$ such that
\[
0\le \tau \le 1, \quad \tau(t)=0\; \text{ for }\;\abs{t-t_0} \ge r^2, \quad \tau(t)=1\; \text{ for } \; \abs{t-t_0} \le (\kappa r)^2, \quad \abs{\tau'}\le C r^{-2}
\]
and also a function $\zeta \in C^\infty_c(\bR^n)$ such that
\[
0\le \zeta \le 1, \quad \zeta(x)=0\;\text{ for }\; \abs{x-x_0} \ge  r,\quad \zeta(x)=1\; \text{ on }\; \abs{x-x_0}\le \kappa r, \quad \abs{D \zeta}\le C r^{-1}.
\]
Denote
\[
\tilde{\vec u}_r(t)=\int_{\Omega_r(x_0)}\zeta^2(x) \vec u(x,t)\,dx \,\Big/ \int_{\Omega_r(x_0)}\zeta^2\,dx.
\]
By testing with 
\[
\eta(x,t)=\zeta^2(x)\tau^2(t) \left\{\vec u(x,t)-\tilde{\vec u}_r(t)\right\}
\]
in \eqref{eqj.1b} and using that
\[
\int_{\Omega_r(x_0)} \zeta^2 (\vec u-\tilde{\vec u}_r(t))\,dx=0,
\]
we obtain for a.e. $s\in (-T,t_0)$ that 
\begin{align*}
0&=\frac{1}{2}\int_{\Omega}\zeta^2\tau^2\Abs{\vec u-\tilde{\vec u}_r(s)}^2\,dx+\int_{-T}^s\!\int_\Omega\zeta^2 \tau^2 A^{\alpha\beta}D_\beta \vec u\cdot D_\alpha \vec u\,dx\,dt\\
&+\int_{-T}^s\!\int_{\Omega}\left\{-\zeta^2\tau\tau' \Abs{\vec u-\tilde{\vec u}_r(t)}^2+2\zeta D_\alpha \zeta \tau^2 A^{\alpha\beta}D_\beta \vec u\cdot \left(\vec u-\tilde{\vec u}_r(t)\right)\right\}\,dx\,dt\\
&+\int_{-T}^s\!\int_{\partial\Omega}\zeta^2\tau^2\theta \vec u\cdot\left(\vec u-\tilde{\vec u}_r(t)\right)\,dS_x\,dt.
\end{align*}
Therefore, by Cauchy's inequality, we have
\begin{multline}		\label{9eq1a}
\esssup_{t_0-r^2<s<t_0}\int_{\Omega}\zeta^2 \tau^2 \Abs{\vec u-\tilde{\vec u}_r(s)}^2 \,dx+\int_{Q'} \zeta^2 \tau^2 \abs{D\vec u}^2\,dx\,dt\\
\le Cr^{-2}\int_{U^{-}_r} \Abs{\vec u-\tilde{\vec u}_r(t)}^2\,dx\,dt+C\int_{S'} \zeta^2 \tau^2 \abs{\vec u}\abs{\vec u-\tilde{\vec u}_r(t)}\,dS_x\,dt.
\end{multline}
We recall the embedding inequality \eqref{eq2.16wine}.
For any $\epsilon \in (0,1)$, we have
\begin{align}
\nonumber
\int_{S'}\zeta^2 \tau^2 \abs{\vec u} \abs{\vec u-\tilde{\vec u}_r(t)}\,dS_x \,dt
&\le \norm{\zeta \tau \vec u }_{\sL_{2(n+1)/(n+2)}(S^{-}_r)} \;\norm{\zeta \tau(\vec u-\tilde{\vec u}_r)}_{\sL_{2(n+1)/n}(S')}\\ 
\nonumber
&\le  C\norm{\zeta \tau \vec u}_{\sL_{2(n+1)/(n+2)}(S^{-}_r)}\;\tri{\zeta \tau (\vec u-\tilde{\vec u}_r)}_{Q'}\\
\label{j27.eq1}
&\le C_\epsilon \norm{\zeta \tau \vec u}^2_{\sL_{2(n+1)/(n+2)}(S^{-}_r)}+\epsilon \tri{\zeta \tau (\vec u-\tilde{\vec u}_r)}_{Q'}^2.
\end{align}
In the above and below, $C_\epsilon$ denotes a constant that depends on $\epsilon$.
Therefore, by combining \eqref{9eq1a} and \eqref{j27.eq1}, we get
\begin{multline}		\label{9eq1c}
\esssup_{t_0-r^2<s<t_0}\int_{\Omega}\zeta^2 \tau^2 \abs{\vec u-\tilde{\vec u}_r(s)}^2\,dx+\int_{Q'}\zeta^2 \tau^2 \abs{D\vec u}^2\,dx\,dt\\
\le Cr^{-2}\int_{U^{-}_r} \abs{\vec u-\tilde{\vec u}_r(t)}^2\,dx\,dt
+C\norm{\zeta \tau \vec u }_{\sL_{2(n+1)/(n+2)}(S^{-}_r)}^2.
\end{multline}
Let $p=2$ if $n\ge 3$ and $p \in (6/5,2)$ if $n=2$, and set
\[
p^*=np/(n-p)\quad \text{and}\quad q=p^*/(p^*-1).
\]
Note that $q=2n/(n+2)$ if $n\ge 3$ and $1<q<3/2$ if $n=2$.

By H\"older's inequality together with properties of $\zeta$ and $\tau$, we get
\begin{align*}
\norm{\zeta \tau \vec u}_{\sL_{2(n+1)/(n+2)}(S^{-}_r)}^2
& \le 2\norm{\zeta\tau(\vec u-\vec u_r)}_{\sL_{2(n+1)/(n+2)}(S^{-}_r)}^2+ 2\abs{\vec u_r}^2  \abs{S^{-}_r}^{(n+2)/(n+1)}\\
&\le Cr^{2+2(n+2)(\frac{1}{2}-\frac{1}{p})} \norm{\zeta(\vec u- \vec u_r)}_{\sL_{p(n-1)/(n-p),p}(S')}^2+
Cr^{(n+2)(1-\frac{2}{q})}\left(\int_{U^{-}_r} \abs{\vec u}^q\right)^{\frac{2}{q}}.
\end{align*}
On the other hand, by trace Sobolev inequality, we get
\begin{align*}
\norm{\zeta(\vec u- \vec u_r)}_{\sL_{p(n-1)/(n-p),p}(S')}^2 
&\le C \left\{\norm{\zeta(\vec u- \vec u_r)}_{\sL_p(Q')}+\norm{D\zeta(\vec u- \vec u_r)}_{\sL_p(Q')} + \norm{\zeta D\vec u}_{\sL_p(Q')}\right\}^2 \\
&\le C r^{2(n+2)(\frac{1}{2}-\frac{1}{p})}\left(\int_{U^{-}_r}\abs{D\vec u}^2+r^{-2}\abs{\vec u-\vec u_r}^2\,dX \right).
\end{align*}
Therefore, we get from the above two inequalities that
\begin{equation}			\label{eq6.15stupid}
\norm{\zeta \tau \vec u}_{\sL_{2(n+1)/(n+2)}(S^{-}_r)}^2\le C r^2 \left\{\int_{U^{-}_r}\abs{D\vec u}^2+r^{-2}\abs{\vec u-\vec u_r}^2\,dX\right\} +Cr^{(n+2)(1-\frac{2}{q})}\left(\int_{U^{-}_r} \abs{\vec u}^{q} \right)^{\frac{2}{q}}.
\end{equation}

\begin{lemma}
For $\epsilon\in (0,1)$, there exists $r_1>0$ such that for $r<\min(r_1,\sqrt{T})$, we have
\begin{multline}		\label{bored}
r^{-2}\int_{U^{-}_{\kappa r}}\abs{\vec u-\tilde{\vec u}_{\kappa r}(t)}^2\,dx\,dt
\le \epsilon \left\{\int_{U^{-}_r}\abs{D\vec u}^2+ r^{-2} \abs{\vec u-\vec u_r}^2\,dX\right\}\\
+C_\epsilon r^{(n+2)(1-\frac{2}{q})}\left\{\int_{U^{-}_r} \abs{\vec u}^q + \abs{D\vec u}^q \,dX \right\}^{\frac{2}{q}}.
\end{multline}
\end{lemma}
\begin{proof}
By H\"older's inequality, \eqref{9eq1c}, and a variant of Poincar\'e inequality, we get
\begin{align*}
\int_{U^{-}_{\kappa r}} \Abs{\vec u-\tilde{\vec u}_r(t)}^2 dx\,dt
&\le \norm{\vec u-\tilde{\vec u}_r}_{\sL_{2,\infty}(U^{-}_{\kappa r})} \norm{\vec u-\tilde{\vec u}_r}_{\sL_{2,1}(U^{-}_{\kappa r})}\\
&\le C\left\{\norm{D\vec u}_{\sL_2(U^{-}_r)}+\norm{\zeta \tau \vec u}_{\sL_{2(n+1)/(n+2)}(S^{-}_r)}\right\} \norm{\vec u-\tilde{\vec u}_r}_{\sL_{2,1}(U^{-}_{\kappa r})}.
\end{align*}
Also, by H\"older's inequality, \eqref{eq5.31cr}, and a variant of Poincar\'e inequality, we have
\begin{align*}
\norm{\vec u-\tilde{\vec u}_r}_{\sL_{2,1}(U^{-}_{\kappa r})}
&\le \int_{t_0-(\kappa r)^2}^{t_0} \norm{\vec u-\tilde{\vec u}_r}_{L^{p^*}(\Omega_{\kappa r})}^{\frac{1}{2}}\norm{\vec u-\tilde{\vec u}_r}_{L^q(\Omega_{\kappa r})}^{\frac{1}{2}}\;dt\\
&\le Cr^{\frac{1}{2}}\int_{t_0-(\kappa r)^2}^{t_0}\left\{ r^{-1}\norm{\vec u-\tilde{\vec u}_r}_{L^p(\Omega_r)}+\norm{D\vec u}_{L^p(\Omega_r)}\right\}^{\frac{1}{2}} \norm{D\vec u}^{\frac{1}{2}}_{L^q(\Omega_{\kappa r})}\;dt\\
&\le Cr^{\frac{1}{2}}\int_{t_0-r^2}^{t_0}\norm{D\vec u}^{\frac{1}{2}}_{L^p(\Omega_r)} \,\norm{D\vec u}^{\frac{1}{2}}_{L^q(\Omega_r)}\;dt\\
&\le Cr^{2+\frac{n}{2p}-\frac{n}{4}-\frac{1}{q}} \norm{D\vec u}^{\frac{1}{2}}_{\sL_2(U^{-}_r)} \norm{D\vec u}^{\frac{1}{2}}_{\sL_q(U^{-}_r)}.
\end{align*}
By combining the above two estimates and using Young's inequality, we get 
\begin{multline*}
r^{-2} \int_{U^{-}_{\kappa r}}\abs{\vec u-\tilde{\vec u}_r(t)}^2\,dx\,dt
\le C r^{\frac{n}{2p}-\frac{n}{4}-\frac{1}{q}} \left\{\norm{D\vec u}_{\sL_2(U^{-}_r)}+\norm{\zeta \tau \vec u}_{\sL_{2(n+1)/(n+2)}(S^{-}_r)}\right\}^{\frac{3}{2}} \norm{D\vec u}^{\frac{1}{2}}_{\sL_q(U^{-}_r)}\\
\le \frac{\epsilon}{4}\left\{\norm{D\vec u}_{\sL_2(U^{-}_r)}+\norm{\zeta \tau \vec u}_{\sL_{2(n+1)/(n+2)}(S^{-}_r)}\right\}^2+ C_\epsilon r^{n+2-\frac{2(n+2)}{q}} \norm{D\vec u}_{\sL_q(U^{-}_r)}^2,
\end{multline*}
where we used that $1/p=1+1/n-1/q$.
Therefore, by \eqref{eq6.15stupid}, we find that for any $\epsilon \in (0,1)$, there exists $r_1$ such that for all $0<r<r_1$, we have
\begin{multline}		\label{eq6.19humid}
r^{-2}\int_{U^{-}_{\kappa r}}\abs{\vec u-\tilde{\vec u}_r(t)}^2\,dx\,dt
\le \epsilon \left\{\int_{U^{-}_r}\abs{D\vec u}^2+ r^{-2} \abs{\vec u-\vec u_r}^2\,dX\right\}\\
+C_\epsilon r^{(n+2)(1-\frac{2}{q})}\left\{\int_{U^{-}_r} \abs{\vec u}^q + \abs{D\vec u}^q \,dX \right\}^{\frac{2}{q}}.
\end{multline}
\\
Note that by the properties of $\zeta$, we have  
\[
c \int_{\Omega_{\kappa r}\times \set{s}} \abs{\vec u-\tilde{\vec u}_{\kappa r}(s)}^2\,dx 
\le  \int_{\Omega_{\kappa r}\times \set{s}} \abs{\vec u-\vec u_{\kappa r}(s)}^2\,dx 
\le \int_{\Omega_r \times \set{s}}  \zeta^2 \abs{\vec u-\tilde{\vec u}_r(s)}^2\,dx. 
\]
Therefore, we get \eqref{bored} from \eqref{eq6.19humid} and the above inequality.
\end{proof} 

By replacing $r$ by $r/\kappa$, we derive the following inequality from \eqref{9eq1c} -- \eqref{bored}:
\begin{multline}			\label{hee}
\int_{U^{-}_{\kappa r}} \abs{D \vec u}^2 \,dX \le 
\epsilon \left\{\int_{U^{-}_{r/\kappa}}\abs{D\vec u}^2+ (r/\kappa)^{-2} \abs{\vec u-\vec u_{r/\kappa}}^2\,dX\right\}\\
+C_\epsilon  r^{(n+2)(1-\frac{2}{q})}\left\{\int_{U^{-}_{r/\kappa}} \abs{\vec u}^q + \abs{D\vec u}^q \,dX \right\}^{\frac{2}{q}}.
\end{multline}
Next, note that
\begin{align}
\nonumber
\int_{U^{-}_{\kappa r}}\abs{\vec u}^2\,dX&\le 2 \int_{U^{-}_{r/\kappa}} \abs{\vec u- \vec u_{r/\kappa}}^2\,dX+ C r^{n+2} \abs{\vec u_{r/\kappa}}^2 \\
\label{dong}
&\le C r^2 \int_{U^{-}_{r/\kappa}} (r/\kappa)^{-2}\abs{\vec u-\vec u_{r/\kappa}}^2\,dX+Cr^{(n+2)(1-\frac{2}{q})}\left(\int_{U^{-}_{r/\kappa}}\abs{\vec u}^q\,dX\right)^{\frac{2}{q}},
\end{align}
Next, by Lemma~\ref{13-lem1} and \eqref{hee}, we get
\begin{align}
\nonumber
r^{-2}\int_{U^{-}_{\kappa r}} \abs{\vec u-\vec u_{\kappa r}}^2\,dX
&\le C\int_{U^{-}_{\kappa r}}\abs{D\vec u}^2\,dX+\int_{U^{-}_{r/\kappa}} \abs{\vec u-\vec u_{r/\kappa}}^2\,dX+Cr^{n+2}\abs{\vec u_{r/\kappa}}^2\\
\nonumber
&\le C(\epsilon+r^2) \left\{\int_{U^{-}_{r/\kappa}}\abs{D\vec u}^2+ (r/\kappa)^{-2} \abs{\vec u-\vec u_{r/\kappa}}^2\,dX\right\}\\
\label{hyun}
&\qquad\qquad +C_\epsilon r^{(n+2)(1-\frac{2}{q})}\left\{\int_{U^{-}_{r/\kappa}} \abs{D\vec u}^q+\abs{\vec u}^q\,dX\right\}^{2/q}.
\end{align}
Therefore, by combining \eqref{hee}, \eqref{dong}, and \eqref{hyun}, we get the result.
\hfill\qedsymbol

\subsection{Proof of Lemma~\ref{13-lem1}}
The proof is an adaptation of that of \cite[Lemma~3]{Struwe}.
Recall the notations \eqref{lcd}.
Let $\chi \in C^\infty_c(B_r(x_0))$ is a smooth cut-off function satisfying 
\[
0\le \chi \le 1, \quad \chi \equiv 1\; \text{ on }\; B_{r/2}(x_0), \quad \abs{D \chi}\le 4r^{-1}.
\]
and denote
\[
\tilde{\vec u}_r(t):=\int_{\Omega_r}\chi \vec u(\cdot,t) \,\Big/ \int_{\Omega_r}\chi,\quad
\tilde{\vec u}_r:=\fint_{b-r^2}^b \tilde{\vec u}_r(s) \,ds,\quad \vec u_r=\fint_{U_r} \vec u\,dX
\]
Note that by testing with $\chi1_{[s,t]} \left(\bar{\vec u}_r(t)-\bar{\vec u}_r(s)\right)$ in \eqref{eqj.1b},  where $s,t \in (a,b)$, we obtain 
\begin{multline*}
0=\left(\int_{\Omega}\chi \vec u(\cdot,t)- \int_{\Omega}\chi \vec u(\cdot,s)\right)\cdot \left(\tilde{\vec u}_r(t)-\tilde{\vec u}_r(s)\right) \\
+\int_s^t\!\!\!\int_\Omega D_\alpha \chi \vec A^{\alpha\beta}D_\beta \vec u\cdot  \left(\tilde{\vec u}_r(t)-\tilde{\vec u}_r(s)\right)\,dX+\int_s^t\!\!\!\int_{\partial \Omega} \chi \theta \vec u \cdot \left(\tilde{\vec u}_r(t)-\tilde{\vec u}_r(s)\right)\,dS_X.
\end{multline*}
Therefore, by using $cr^n \le \abs{\Omega_r}$ and the trace theorem, we get
\begin{align*}
\Abs{\tilde{\vec u}_r(t)-\tilde{\vec u}_r(s)}^2 
&\le C r^{-n} \Abs{\tilde{\vec u}_r(t)-\tilde{\vec u}_r(s)}
\left(\int_s^t\!\!\!\int_{\Omega_r} r^{-1}\abs{D\vec u}\,dX+ \norm{\theta}_\infty \int_s^t\!\!\!\int_{\partial\Omega}\abs{\chi \vec u}\,dS_X\right)\\
&\le C r^{-n} \Abs{\tilde{\vec u}_r(t)-\tilde{\vec u}_r(s)}
\left(r^{-1}\int_s^t\!\!\!\int_{\Omega_{r}}\abs{D\vec u}\,dX+\int_s^t\!\!\int_{\Omega_r}r^{-1}\abs{\vec u}+\abs{D\vec u}\,dX\right)
\end{align*}
and thus, by H\"older's inequality, for $s,t \in (b-r^2, b)$ and $r<\min(\sqrt{b-a}, \diam \Omega)$, we have
\begin{equation}		\label{13-eq2d}
\Abs{\tilde{\vec u}_r(t)-\tilde{\vec u}_r(s)}
\le C r^{-n/2}\left(\norm{\vec u}_{\sL_2(U^{-}_r)}+\norm{D\vec u}_{\sL_2(U^{-}_r)}\right).
\end{equation}
Now, we note that  
\begin{align}
\nonumber
\int_{U^{-}_r}\Abs{\vec u-\vec u_r}^2\,dX
&\le \int_{U^{-}_r}\Abs{\vec u-\tilde{\vec u}_r}^2\,dX\\
\nonumber
&\le 2\left(\int_{b-r^2}^b\!\int_{\Omega_r} \Abs{\vec u(x,t)-\bar{\vec u}_r(t)}^2+\Abs{\tilde{\vec u}_r(t)-\tilde{\vec u}_r}^2\,dx\,dt\right)\\
\label{13-eq2a}
&\le C r^2\int_{U^{-}_r}\abs{D\vec u}^2 \,dX + C \int_{U^{-}_r}\Abs{ \fint_{b-r^2}^b \left( \tilde{\vec u}_r(t)-\tilde{\vec u}_r(s)\right)\,ds}^2\,dX,
\end{align}
where we used a variant of Poincar\'e's inequality.
Therefore, by combining \eqref{13-eq2d} and \eqref{13-eq2a}, we find that 
\begin{align*}
\int_{U^{-}_r}\Abs{\vec u-\vec u_r}^2 \,dX
&\le C r^2 \int_{U^{-}_r}\abs{D\vec u}^2 \,dX + C r^{-n}\left(\norm{\vec u}_{\sL_2(U^{-}_r)}+\norm{D\vec u}_{\sL_2(U^{-}_r)}\right)^2 \abs{U^{-}_r}\\
&\le C r^2 \int_{U^{-}_r} \abs{\vec u}^2 + \abs{D\vec u}^2 \,dX.
\end{align*}
The proof is complete.
\hfill\qedsymbol

\subsection{Proof of Lemma~\ref{er.lem2}}
Let $\vec u$ be a weak solution of \eqref{er.eq1}.
We shall show that for any $\delta \in (0,1)$, there exist constants $C>0$, and $r_0>0$ such that for any $x_0\in \overline \Omega$ and $0<r<r_0$, we have 
\begin{multline}		\label{er.eq5}
\int_{\Omega_{r/2}(x_0)}\abs{D\vec u}^2+\abs{\vec u}^2\,dx\le \delta \int_{\Omega_r(x_0)}\abs{D\vec u}^2+\abs{\vec u}^2\,dx\\
+C r^{n(1-2/q)} \left(\int_{\Omega_{r}(x_0)}\abs{D\vec u}^q+\abs{\vec u}^q\,dx\right)^{2/q}+Cr^2\int_{\Omega_r(x_0)}\abs{\vec f}^2\,dx,
\end{multline}
where $q=2n/(n+2)$ if $n\ge 3$ and $1<q<2$ if $n=2$.
Then, the inequality \eqref{a1} will follow from a version of Gehring's lemma \cite[Proposition 1.1, p. 122]{Gi83}.
Indeed, set
\[
g(x)= \left\{\abs{D\vec u}^2+ \abs{\vec u}^2\right\}^{\frac{q}{2}}1_\Omega(x)\quad\text{and}\quad F(x)= \norm{\vec f}_{L^n(\Omega)}^q 1_\Omega(x).
\]
Then by \eqref{er.eq5}, we have
\[
\int_{B_{r/2}}g^{2/q}\,dx\le \delta \int_{B_r}g^{2/q}\,dx
+C r^n \left(\fint_{B_r}g \,dx\right)^{2/q}+C\int_{B_r} F^{2/q}\,dx,
\]
and thus, \eqref{a1} will follow.
Let $\zeta$ be a smooth cut-off function satisfying
\[
0\le \zeta \le 1, \quad \zeta \equiv 1\; \text{ on }\; B_{r/2}(x_0),\quad \zeta \equiv 0\; \text{ on }\; \bR^n\setminus B_r(x_0), \quad \abs{D \zeta}\le 8r^{-1},
\]
where $0<r<\diam \Omega$, and denote 
\[
\tilde {\vec u}_r:=\int_{\Omega_r}\zeta^2 \vec u\; \Big/ \int_{\Omega_r} \zeta^2.
\]
Then by testing with $\zeta^2 (\vec u-\tilde{\vec u}_r)$ in \eqref{er.eq1}, we find that 
\begin{multline*}
\int_\Omega \zeta^2 A^{\alpha\beta}D_\beta\vec u\cdot D_\alpha \vec u
+\int_\Omega 2\zeta D_\alpha \zeta A^{\alpha\beta}D_\beta \vec u\cdot (\vec u-\tilde{\vec u}_r)
+\int_{\partial \Omega} \zeta^2\theta \vec u \cdot (\vec u-\tilde{\vec u}_r)\\
=\int_\Omega \zeta^2 \vec f\cdot (\vec u-\tilde{\vec u}_r),
\end{multline*}
and thus, by Cauchy's inequality, we get
\begin{equation}			\label{eq6.24capu}
\int_\Omega \zeta^2 \abs{D\vec u}^2\le Cr^{-2}\int_{\Omega_r} \Abs{\vec u-\tilde{\vec u}_r}^2+Cr^2\int_{\Omega_r}\abs{\vec f}^2+C \norm{\theta}_\infty \int_{\partial \Omega} \zeta^2  \abs{\vec u}\abs{\vec u-\tilde{\vec u}_r}.
\end{equation}
By the trace Sobolev inequality and Cauchy's inequality, for any $\epsilon \in (0,1)$ we have
\begin{align}
\nonumber
\int_{\partial \Omega} \zeta^2 \abs{\vec u}\abs{\vec u-\tilde{\vec u}_r}
&\le C\norm{\zeta\vec u}_{L^2(\partial \Omega)}\norm{\zeta (\vec u-\tilde{\vec u}_r)}_{W^{1,2n/(n+1)}(\Omega)}\\
\nonumber
&\le C r^{1/2}\norm{\zeta\vec u}_{L(\partial \Omega)}\norm{\zeta (\vec u-\tilde{\vec u}_r)}_{W^{1,2}(\Omega)}\\
\nonumber
&\le \epsilon \norm{\zeta (\vec u-\tilde{\vec u}_r)}_{W^{1,2}(\Omega)}^2+ C_\epsilon r\left(\norm{\zeta (\vec u-\tilde{\vec u}_r)}_{L^2(\partial \Omega)}^2+r^{n-1} \abs{\tilde{\vec u}_r}^2\right)\\
\nonumber
&\le (\epsilon+C_\epsilon r) \norm{\zeta (\vec u-\tilde{\vec u}_r)}_{W^{1,2}(\Omega)}^2+ C_\epsilon r^n\abs{\tilde{\vec u}_r}^2\\
\label{jer.eq1b}
&\le (C\epsilon+C_\epsilon r)\left\{r^{-2}\norm{\vec u-\tilde{\vec u}_r}_{L^2(\Omega_r)}^2+\norm{D\vec u}_{L^2(\Omega_r)}^2\right\}+ C_\epsilon r^n\abs{\tilde{\vec u}_r}^2,
\end{align}
where $C_\epsilon$ is a constant that depends on $\epsilon$ as well.
By \eqref{eq6.24capu} and \eqref{jer.eq1b}, we get
\[
\int_{\Omega_{r/2}}\abs{D\vec u}^2
\le \frac{C_\epsilon}{r^2}\int_{\Omega_r} \abs{\vec u-\tilde{\vec u}_r}^2 +(C\epsilon+C_\epsilon r) \int_{\Omega_r} \abs{D\vec u}^2+Cr^2\int_{\Omega_r}\abs{\vec f}^2+C_\epsilon r^n\abs{\tilde{\vec u}_r}^2.
\]
Therefore, by choosing $\epsilon$ and then $r_1$ so small that for all $r \in (0, r_1)$, we have
\begin{equation}		\label{jer.eq1d}
\int_{\Omega_{r/2}}\abs{D\vec u}^2
\le \frac{C}{r^2}\int_{\Omega_r}\abs{\vec u-\tilde{\vec u}_r}^2 +\frac{\delta}{4} \int_{\Omega_r} \abs{D\vec u}^2+Cr^2\int_{\Omega_r}\abs{\vec f}^2+C r^n\abs{\tilde{\vec u}_r}^2.
\end{equation}

Now, we take $p=2$ if $n\ge 3$ and $p \in (1,2)$ if $n=2$ and set
\[
p^*=np/(n-p)\quad \text{and}\quad q=p^*/(p^*-1).
\]
Note that $q=2n/(n+2)$ if $n\ge 3$ and $1<q<2$ if $n=2$.

By using H\"older's inequality and a variant of Poincar\'e's inequality, we have  
\begin{align}
\nonumber
\int_{\Omega_r}\abs{\vec u-\tilde{\vec u}_r}^2
&\le \norm{\vec u-\tilde{\vec u}_r}_{L^{p^*}(\Omega_r)} \norm{\vec u-\tilde{\vec u}_r}_{L^{q}(\Omega_r)} 
\le Cr \norm{D\vec u}_{L^p(\Omega_r)}\norm{D\vec u}_{L^q(\Omega_r)}\\
\nonumber
&\le Cr^{1+n/p-n/2} \norm{D\vec u}_{L^2(\Omega_r)}\norm{D\vec u}_{L^q(\Omega_r)},
\end{align}
and thus, by Cauchy's inequality, we get for any $\epsilon\in (0,1)$ that  
\begin{equation}		\label{er.eq1d}
\int_{\Omega_r}\abs{\vec u-\tilde{\vec u}_r}^2
\le \epsilon r^2\int_{\Omega_r}\abs{D\vec u}^2+C_\epsilon r^{n(1-2/q)+2}\left(\int_{\Omega_r}\abs{D\vec u}^q\right)^{2/q}.
\end{equation}
Therefore, we conclude from \eqref{jer.eq1d} and \eqref{er.eq1d} that for all $r \in (0, r_1)$, we have
\[
\int_{\Omega_{r/2}}\abs{D\vec u}^2\le  \frac{\delta}{2} \int_{\Omega_r}\abs{D\vec u}^2
+C r^{n(1-2/q)}\left(\int_{\Omega_r}\abs{D\vec u}^q+\abs{\vec u}^q\right)^{2/q}+Cr^2\int_{\Omega_r}\abs{\vec f}^2,
\]
where we used the fact that 
\[
\abs{\tilde{\vec u}_r}^2\le C r^{-2n/q}\left(\int_{\Omega_r} \abs{\vec u}^q\right)^{2/q}.
\]
Finally, we apply the inequality
\[
\int_{\Omega_{r/2}}\abs{\vec u}^2\le 2\int_{\Omega_r}\abs{\vec u-\tilde{\vec u}_r}^2+2\abs{\Omega_r} \abs{\tilde{\vec u}_r}^2\le  Cr^2 \int_{\Omega_r}\abs{D\vec u}^2+Cr^{n(1-2/q)}\left(\int_{\Omega_r}\abs{\vec u}^q\right)^{2/q}
\]
to conclude that for all $0<r<r_0\le r_1$, we have
\[
\int_{\Omega_{r/2}} \left(\abs{D\vec u}^2+\abs{\vec u}^2\right) \le (\delta/2+C r_0^2) \int_{\Omega_r}\abs{D\vec u}^2
+C r^{n(1-2/q)}\left(\int_{\Omega_r}\abs{D\vec u}^q+\abs{\vec u}^q \right)^{2/q}+Cr^2\int_{\Omega_r}\abs{\vec f}^2,
\]
which clearly implies \eqref{er.eq5}.
\hfill\qedsymbol

\subsection{Proof of Lemma~\ref{lem5.6er}}
Without loss of generality, we assume that $a=-T$ and $b=0$ so that $Q=\Omega\times (-T,0)$ and $S=\partial\Omega \times (-T,0)$.
The proof is an adaptation of that of \cite[Lemma~4.2]{Kim}.
First, note that \eqref{b9} follows from the energy inequality since we assume the condition (H1).
To prove the rest, we claim that $\vec u_t$ satisfies
\begin{equation}		\label{jk2}
\norm{\vec u_t}_{\sL_2(\Omega\times(-r^2,0))}\le C(R-r)^{-1} \left(\norm{D\vec u}_{\sL_2(\Omega\times(-R^2,0))}+\norm{\vec u}_{\sL_2(\Omega\times (-R^2,0))}\right)
\end{equation}
for all  $0<r<R<\min(\sqrt{T}, \diam\Omega)$.
Take the above inequality for now.
By $t$-independence of the operator, we find that $\vec u_t$ is also a weak solution of \eqref{eq5.38pvd}.
Therefore, by the energy inequalities (cf. \eqref{app.eq1c} -- \eqref{eq5.13hu}) and \eqref{jk2}, we get that
\begin{align*}
\sup_{-(r/2)^2\le s\le 0}\int_{\Omega}\abs{\vec u_t(x,s)}^2\,dx & \le \frac{C}{r^2}\int_{\Omega\times(-(3r/4)^2,0)}\abs{\vec u_t}^2\,dX\\
&\le \frac{C}{r^4}\int_{\Omega\times(-(7r/8)^2,0)}\abs{\vec u}^2+\abs{D\vec u}^2 \,dX \le \frac{C}{r^6}\int_{\Omega\times (-r^2,0)}\abs{\vec u}^2\,dX,
\end{align*}
where we used that $r\le \diam \Omega$.
We have established \eqref{jk6}.
To prove \eqref{deq1}, fix a function $\tau \in C^\infty_c(\bR)$ such that 
\[
0\le \tau\le 1, \quad  \tau(t) =0\; \text{ for }\; \abs{t}\ge r^2, \quad \tau(t)=1 \;\text{ for } \;\abs{t} \le (r/2)^2, \quad \abs{\tau'}\le 8 r^{-2}.
\]
On each slice $\Omega\times \set{s}$, where $-T<s<0$, we have
\begin{multline}		\label{eq6.44snow}
0=\int_{\Omega \times \set{s}} \left(\vec u_t - D_\alpha(A^{\alpha\beta} D_\beta \vec u) \right)\cdot \tau^2 \vec u \,dx= \int_{\Omega \times \set{s}} \tau^2 \vec u_t \cdot \vec u\,dx + \int_{\partial\Omega \times \set{s}} \tau^2 \theta \vec u\cdot \vec u\,dS_x \\
+\int_{\Omega \times \set{s}} \tau^2 A^{\alpha \beta} D_\beta \vec u \cdot D_\alpha \vec u\,dx.
\end{multline}
Then, by Cauchy's inequality and the trace theorem (recall $\tau =\tau(t)$) we get
\begin{align*}
\int_{\Omega\times \set{s}} \!\!\!\tau^2\abs{D\vec u}^2 \,dx
& \le C\int_{\Omega\times \set{s}}\!\!\!\tau^2\abs{\vec u}\abs{\vec u_t}\,dx + C \int _{\partial \Omega \times \set{s}} \!\!\!  \tau^2 \abs{\vec u}^2 \,dS_x  \\
&\le C\int_{\Omega\times \set{s}} \!\!\! \tau^2\abs{\vec u}\abs{\vec u_t}\,dx+ C \int_{\Omega\times \set{s}}\!\!\!  2\tau^2 \abs{\vec u \cdot D\vec u} + \tau^2 \abs{\vec u}^2 \,dx \\
&\le C r^2\int_{\Omega\times \set{s}} \!\!\! \tau^2\abs{\vec u_t}^2\,dx+C(1+ r^{-2})\int_{\Omega\times\set{s}} \!\!\!\tau^2\abs{\vec u}^2\,dx + \frac{1}{2} \int_{\Omega\times \set{s}} \!\!\!  \tau^2 \abs{D\vec u}^2 \,dx,
\end{align*}
and thus, by using \eqref{b9} and \eqref{jk6}, we obtain (recall $r\le \diam \Omega$)
\[
\sup_{-(r/2)^2 <s<0} \int_\Omega \abs{D\vec u(x,s)}^2\,dx \le Cr^{-4}\int_{\Omega\times (-r^2,0)}\abs{\vec u}^2\,dX,
\]
which establishes \eqref{deq1}.

It only remains us to prove the claim \eqref{jk2}, the proof of which is a mere adaptation of that of \cite[Lemma~4.1]{Kim}.
For any $0<r<\rho<\min(\sqrt{T},\diam \Omega)$, let $\zeta=\zeta(t)$ be a smooth function on $\bR$ such that
\[
0\le \zeta\le 1, \quad  \zeta(t) =0\; \text{ for }\; \abs{t}\ge \rho^2, \quad \zeta(t)=1 \;\text{ for } \;\abs{t} \le r^2, \quad \abs{\zeta'}\le 2 (\rho-r)^{-2}.
\]
Similar to \eqref{eq6.44snow}, on each slice $\Omega\times \set{s}$, where $-T<s<0$, we have
\begin{multline*}
0=\int_{\Omega \times \set{s}} \left(\vec u_t - D_\alpha(A^{\alpha\beta} D_\beta \vec u) \right)\cdot \zeta^2 \vec u_t \,dx= \int_{\Omega \times \set{s}} \zeta^2 \abs{\vec u_t}^2\,dx + \int_{\partial\Omega \times \set{s}} \zeta^2 \theta \vec u\cdot \vec u_t\,dS_x \\
+\int_{\Omega \times \set{s}} \zeta^2 A^{\alpha \beta} D_\beta \vec u \cdot D_\alpha \vec u_t\,dx.
\end{multline*}
Therefore, by Cauchy's inequality, we get
\begin{multline}		\label{jk3}
\int_{\Omega\times \set{s}}\zeta^2\abs{\vec u_t}^2\,dx
\le C\int_{\Omega\times \set{s}}\zeta^2\abs{D\vec u}\abs{D\vec u_t}\,dx+C\int_{\partial \Omega\times \set{s}}\zeta^2\abs{\vec u}\abs{\vec u_t}\,dS_x\\
\le \epsilon\int_{\Omega\times \set{s}}\zeta^2\abs{D\vec u_t}^2\,dx+\frac{C}{\epsilon}\int_{\Omega\times \set{s}}\zeta^2\abs{D\vec u}^2\,dx+C\int_{\partial\Omega\times \set{s}}\zeta^2\abs{\vec u}\abs{\vec u_t}\,dS_x.
\end{multline}
Note that by H\"older's inequality and the trace theorem, we have
\begin{multline*}
\int_{\partial\Omega\times \set{s}}\zeta^2\abs{\vec u}\abs{\vec u_t} \,dS_x \le    C\norm{\zeta\vec u(\cdot,s)}_{W^{1,2}(\Omega)}\norm{\zeta\vec u_t(\cdot,s)}_{W^{1,2}(\Omega)} \\
\le  C \left( \norm{\zeta\vec u(\cdot,s)}_{L^2(\Omega)}+\norm{\zeta D\vec u(\cdot,s)}_{L^2(\Omega)}\right) \left( \norm{\zeta\vec u_t(\cdot,s)}_{L^2(\Omega)}+\norm{\zeta D\vec u_t(\cdot,s)}_{L^2(\Omega)}\right).
\end{multline*}
By the above inequality and Cauchy's inequality, we have 
\begin{multline}		\label{jk4}
\int_{\partial\Omega \times \set{s}}\zeta^2\abs{\vec u}\abs{\vec u_t} \,dS_x \le 2\epsilon'\norm{\zeta\vec u_t(\cdot,s)}^2_{L^2(\Omega)}+\frac{C}{\epsilon'}\norm{\zeta\vec u(\cdot,s)}^2_{L^2(\Omega)}+
\frac{C}{\epsilon'}\norm{\zeta D\vec u(\cdot,s)}_{L^2(\Omega)}^2\\
+2\epsilon\norm{\zeta D\vec u_t(\cdot,s)}^2_{L^2(\Omega)}
+\frac{C}{\epsilon}\norm{\zeta\vec u(\cdot,s)}^2_{L^2(\Omega)}
+\frac{C}{\epsilon}\norm{\zeta D\vec u(\cdot,s)}^2_{L^2(\Omega)}. 
\end{multline}
Therefore,  by combining \eqref{jk3} and \eqref{jk4} and integrating over $(-T, 0)$, we get
\begin{multline}		\label{jk5}
\int_{Q}\zeta^2\abs{\vec u_t}^2\,dx \le 3\epsilon \int_{Q}\zeta^2\abs{D\vec u_t}^2\,dX
+2\epsilon'\int_{Q}\zeta^2\abs{\vec u_t}^2\,dX\\
+C\left(\frac{1}{\epsilon}+\frac{1}{\epsilon'}\right)\int_{Q}\zeta^2 \left(\abs{\vec u}^2+ \abs{D\vec u}^2\right) dX,
\end{multline}
Since $\vec u_t$ also satisfies \eqref{eq5.38pvd}, Caccioppoli type inequality together with the property of $\zeta$ yield that 
\[
\int_{Q}\zeta^2\abs{D\vec u_t}^2 \,dX \le \frac{C'}{(\rho-r)^2}\int_{\Omega\times (-R^2,0)}\abs{\vec u_t}^2\,dX,
\]
and thus, we derive from \eqref{jk5} that 
\begin{multline*}
\int_{\Omega\times (-r^2,0)}\abs{\vec u_t}^2\,dX
\le \left(\frac{3\epsilon C'}{(\rho-r)^2}+2\epsilon'\right)\int_{\Omega\times (-\rho^2,0)}\abs{\vec u_t}^2 \,dX\\
+C\left(\frac{1}{\epsilon}+\frac{1}{\epsilon'}\right)\int_{\Omega\times (-\rho^2,0)} \abs{\vec u}^2 +\abs{D\vec u}^2\,dX.
\end{multline*}
If we set $\epsilon=(\rho-r)^2/12C'$ and $\epsilon'=1/8$ in the above, we get 
\[
\int_{\Omega\times (-r^2,0)}\abs{\vec u_t}^2 \,dX \le \frac{1}{2}\int_{\Omega\times (-\rho^2,0)}\abs{\vec u_t}^2 \,dX+\frac{C}{(\rho-r)^2}\int_{\Omega\times (-\rho^2,0)}\abs{D\vec u}^2+ \abs{\vec u}^2 \,dX,
\]
where we used that $\rho-r \le \diam \Omega$.
Then by using an iteration method (see \cite[Lemma~3.1, p. 161]{Gi83}), we obtain \eqref{jk2} from the above inequality.
\hfill\qedsymbol

\begin{acknowledgment}
We thank  Fritz Gesztesy for helpful discussion and correspondence.
Jongkeun Choi is partially supported by the Yonsei University Research Fund No. 2014-12-0003.
\end{acknowledgment}


\end{document}